\documentclass{article}
\usepackage[utf8]{inputenc}
\usepackage[english]{babel}
\usepackage{amsmath}
\usepackage{wasysym}
\usepackage[style=alphabetic]{biblatex}
\usepackage{chngcntr}
\usepackage{listings}
\usepackage{mathtools}

\makeatletter
\DeclareRobustCommand\widecheck[1]{{\mathpalette\@widecheck{#1}}}
\def\@widecheck#1#2{%
    \setbox\z@\hbox{\m@th$#1#2$}%
    \setbox\tw@\hbox{\m@th$#1%
       \widehat{%
          \vrule\@width\z@\@height\ht\z@
          \vrule\@height\z@\@width\wd\z@}$}%
    \dp\tw@-\ht\z@
    \@tempdima\ht\z@ \advance\@tempdima2\ht\tw@ \divide\@tempdima\thr@@
    \setbox\tw@\hbox{%
       \raise\@tempdima\hbox{\scalebox{1}[-1]{\lower\@tempdima\box
\tw@}}}%
    {\ooalign{\box\tw@ \cr \box\z@}}}
\makeatother
\usepackage{floatrow}
\usepackage{amssymb}
\usepackage[left=3.5cm, right=3.5cm, top = 3cm, bottom = 3cm]{geometry}
\usepackage{mathrsfs}
\usepackage{lipsum}
\usepackage{amsfonts}
\usepackage[parfill]{parskip}
\usepackage{tikz-cd}
\usepackage{calc}
\usepackage{accents}
\usepackage{graphicx,wrapfig}
\usepackage{hyperref}
\usepackage{amsthm}\usepackage{stmaryrd}
\usepackage{MnSymbol}
\usetikzlibrary{arrows}
\usepackage[T1]{fontenc}
\usepackage[bbgreekl]{mathbbol}
\usepackage{etoolbox}

\newtheorem{thm}{Theorem}[section]
\newtheorem{defn}[thm]{Definition}
\newtheorem{lma}[thm]{Lemma}

\newtheorem{cor}[thm]{Corollary}
\newtheorem{prop}[thm]{Proposition}
\newtheorem{ex}[thm]{Example}
\newtheorem{rem}[thm]{Remark}

\newcommand{\lspan}{\operatorname{span}}
\newcommand{\clspan}{\overline{\operatorname{span}}}
\newcommand{\Mor}{\operatorname{Mor}}
\newcommand{\K}{\mathcal{K}}
\newcommand{\B}{\mathcal{B}}

\newcommand{\G}{\mathbb{G}}

\newcommand{\ovot}{\mathbin{\overline{\otimes}}}
\newcommand{\Irr}{\operatorname{Irr}}

\renewcommand{\H}{\mathcal{H}}

\newcommand{\ev}{\operatorname{ev}}
\newcommand{\Ad}{\operatorname{Ad}}
\newcommand{\Tr}{\operatorname{Tr}}
\newcommand{\wG}{\widehat{\mathbb{G}}}

\newcommand{\id}{\operatorname{id}}
\newcommand{\Di}{\mbox{}\hfill\Diamond} 

\allowdisplaybreaks

\title{Actions of compact and discrete quantum groups on operator systems}

\makeatletter
\renewcommand*\author[1]{%
  \stepcounter{author}%
  \ifnum\c@author=1
    \gdef\@author{#1}%
  \else
    \xdef\@author{\unexpanded\expandafter{\@author\and#1}}%
  \fi
  \csgdef{author@\the\c@author}{#1}}
\newcommand*\email[1]{%
  \csgdef{email@\the\c@author}{#1}}
\newcommand*\address[1]{%
  \csgdef{address@\the\c@author}{#1}}
\AtEndDocument{%
  \xdef\author@count{\the\c@author}%
  \c@author=1
  \print@authors}
\newcommand*\print@authors{%
  \ifnum\c@author>\author@count
  \else
    \print@author{\the\c@author}%
    \advance\c@author by 1
    \expandafter\print@authors
  \fi}
\newcommand*\print@author[1]{%
  \par\medskip
  \begin{tabular}{@{}l@{}}%
    \textsc{\csuse{author@#1}}\\
    \csuse{address@#1}\\
    \textit{E-mail address}:
    \href{mailto:\csuse{email@#1}}{\csuse{email@#1}}
  \end{tabular}}
\makeatother

\author{Joeri De Ro}
\address{Vrije Universiteit Brussel}
\email{joeri.ludo.de.ro@vub.be}

\author{Lucas Hataishi}
\address{University of Oslo}
\email{lucasyh@math.uio.no}

\date{}

\begin{document}

\maketitle

\begin{abstract}
    \noindent We introduce the notion of an action of a discrete or compact quantum group on an operator system, and study equivariant operator system injectivity. We then prove a duality result that relates equivariant injectivity with dual injectivity on associated crossed products. As an application, we give a description of the equivariant injective envelope of the reduced crossed product built from an action of a discrete quantum group on an operator system. 
\end{abstract}

\section{Introduction}

The Furstenberg boundary $\partial_F \Gamma$ associated to a discrete group $\Gamma$ fits naturally in Hamana's theory of $\Gamma$-injective envelopes \cite{Ham2} by observing that $C(\partial_F \Gamma)= I_{\Gamma}(\mathbb{C})$. This key insight has led to important advances in understanding $C^*$-simple discrete groups \cite{KenKal}, and is very suitable to generalisation to the quantum setting. This was first explored in \cite{KKSV} and later in \cite{HHN}. A significant part of these papers is devoted to understanding the notion of equivariant injectivity of a (unital) $C^*$-algebra on which a certain quantum group acts. 

In this paper, we continue the study of equivariant injectivity for actions of compact and discrete quantum groups. Notions of (equivariant) injectivity are naturally studied in the context of operator systems, so we first introduce the notion of an action of a compact and a discrete quantum group on an operator system, leading us to the notion of equivariant operator system. This offers a unifying and natural framework in which the notion of equivariant injectivity, and by extension other dynamical properties of compact and discrete quantum groups, can be studied. The choice to work with compact and discrete quantum groups instead of with general locally compact quantum groups was a deliberate one: the proofs of the main results in this paper rely on the specific structure of compact quantum groups and their duals. Nevertheless, some of the definitions in this paper can be immediately given on the level of locally compact quantum groups, but we refrained from doing so to keep things more concrete and less technical.

We now give an overview of the content of this paper.

In \emph{Section 2}, we explain our notations and conventions and recall some theory relevant for our purposes. First, we discuss operator systems and tensor products of operator systems.  Special attention is given to the Fubini tensor product of operator systems, which will play a crucial role in this paper. Next, we briefly review the theory of compact and discrete quantum groups. The framework for quantum groups we use in this paper is as follows: we start from a compact quantum group $\G$ and consider discrete quantum groups as the duals of compact quantum groups. The dual discrete quantum group of $\G$ will be denoted by $\wG$.

In \emph{Section 3}, we develop the basic theory of actions of compact and discrete quantum groups on operator systems. If $\G$ is a compact quantum group, we will consider two notions of an operator system $X$ with a $\G$-action $X \curvearrowleft \G$:

\begin{enumerate}
    \item A $\G$-$C^*$-operator system (Definition \ref{Gcstar}), which should be thought of as a version of the notion of an action of a compact quantum group on a $C^*$-algebra in the setting of operator systems.
    \item A $\G$-$W^*$-operator system (Definition \ref{GWstar}), which should be thought of as a version of the notion of an action of a compact quantum group on a von Neumann algebra in the setting of operator systems.
\end{enumerate} 

The notion of $\G$-$C^*$-injectivity for $\G$-$C^*$-operator systems is introduced, as well as the notion of $\G$-$W^*$-injectivity for $\G$-$W^*$-operator systems (Definition \ref{Ginjectivity}). 

Next, the notion of a $\wG$-operator system  is introduced, which consists of an operator system $X$ with an action $X\stackrel{\alpha}\curvearrowleft \wG$ (Definition \ref{action DQG}). We also define and investigate the notion of $\wG$-injectivity for a $\wG$-operator system. With respect to the $\wG$-operator system $X\stackrel{\alpha}\curvearrowleft\wG$, we define the reduced crossed product $X\rtimes_{\alpha,r}\wG$ (Definition \ref{reduced}) and the Fubini crossed product $X\rtimes_{\alpha, \mathcal{F}}\wG$ (Definition \ref{Fubini}). We show that $X\rtimes_{\alpha,r}\wG$ has a canonical $\G$-$C^*$-operator system structure and that $X\rtimes_{\alpha, \mathcal{F}} \wG$ has a canonical $\G$-$W^*$-operator system structure and study these equivariant operator systems (Theorem \ref{important action}). 

In \emph{Section 4}, we prove the main result of this paper:

\textbf{Theorem \ref{main result}.} \textit{If $X$ is an operator system with an action $X\stackrel{\alpha}\curvearrowleft\wG$, then the following are equivalent:
\begin{enumerate}
    \item $X$ is $\wG$-injective.
    \item $X\rtimes_{\alpha, r}\wG$ is $\G$-$C^*$-injective.
    \item $X\rtimes_{\alpha, \mathcal{F}}\wG$ is $\G$-$W^*$-injective.
\end{enumerate}
Moreover, if $\G$ is of Kac-type, then $X$ is $\wG$-injective if and only if $X\rtimes_{\alpha,\mathcal{F}} \wG$ is injective as an operator system.}

Note that this result recovers the original result \cite[II, Lemma 3.1]{Ham3} in the classical case, but that it also extends it to the equivariant setting for the crossed products.

Applying the main result to the situation where $\mathbb{C}\curvearrowleft\wG$ leads to the following characterisation of amenability of $\wG$ in terms of equivariant injectivity of associated operator algebras and their canonical $\G$-actions:

\textbf{Corollary \ref{amenability}.} \textit{The following statements are equivalent:
    \begin{enumerate}
        \item $\wG$ is amenable.
        \item $\mathscr{L}^\infty(\G)$ is $\G$-$W^*$-injective.
        \item $C_r(\G)$ is $\G$-$C^*$-injective.
            \end{enumerate}}

Note that in \cite{Crann}, amenability of $\wG$ was characterised using the notion of $1$-injectivity of $\mathscr{L}^\infty(\G)$ as an operator $L^1(\G)$-module, which somewhat resembles one of the equivalences in this result.

In \emph{Section 5}, we prove that if $X$ is a $\G$-$C^*$-operator system, then $X$ admits a $\G$-$C^*$-injective envelope $I_\G^{C^*}(X)$ and that if $X$ is a $\wG$-operator system, then $X$ admits a $\wG$-injective envelope $I_{\wG}(X)$. We then relate these injective envelopes by the following duality result:

\textbf{Theorem \ref{description}.} \textit{Let $X$ be a $\wG$-operator system. Then
$$I_{\wG}(X)\rtimes_r \wG = I_{\G}^{C^*}(X\rtimes_r \wG).$$}

\section{Preliminaries}

\subsection{Tensor products of operator systems}

\textbf{Operator systems.} \cite{Paulsen} An operator system $X$ is a norm-closed, self-adjoint linear subspace $X\subseteq B(\H)$ where $\H$ is a Hilbert space such that $X$ contains the unit of $B(\H)$. We write $X_*:= \{\omega\vert_X: \omega \in B(\H)_*\}$. We will abbreviate (unital) completely positive map by (u)cp map and (unital) completely isometric map by (u)ci map. If $Y\subseteq X$ are operator systems, then a map $\varphi: X \to Y$ is called ucp conditional expectation if $\varphi$ is a ucp map such that $\varphi(y) = y$ for all $y\in Y$. An operator system $I$ is called injective if for every ucp map $\varphi: X \to I$ and every uci map $\iota: X\to Y$, there is a ucp map $\Phi: Y \to I$ such that $\Phi \iota = \phi$. Arveson's extension theorem says that $B(\H)$ is an injective operator system for every Hilbert space $\H$.

Given two operator systems, there are several kinds of tensor products that will be important to us. The usual algebraic tensor product of vector spaces will be denoted by $\odot$. The letters $\H$ and $\K$ will always denote Hilbert spaces and their Hilbert space tensor product is denoted by $\H \otimes \K$. We assume that inner products are linear in the second factor.

\textbf{Spatial tensor product}. \cite[Chapter 12]{Paulsen} Let $X\subseteq B(\H),Y\subseteq B(\K)$ be operator systems. We define the \emph{spatial tensor product} of $X$ and $Y$ to be the closed linear span of $X\odot Y$ inside $B(\H \otimes \K)$. If $X,Y$ are $C^*$-algebras, this agrees with the minimal tensor product of $C^*$-algebras. If $\phi_1: X_1\to Y_1$ and $\phi_2: X_2\to Y_2$ are ucp maps, there is a unique ucp map $\phi_1\otimes \phi_2: X_1\otimes Y_1\to X_2\otimes Y_2$ that extends $\phi_1\odot \phi_2: X_1\odot Y_1\to X_2\odot Y_2$. 

\textbf{Fubini tensor product}. \cite{Ham3, Ham1} Let $X\subseteq B(\H),Y\subseteq B(\K)$ be operator systems. We define the \emph{Fubini tensor product}
$$X \ovot Y = \{z \in B(\H\otimes \K): (\omega \ovot \id)(z) \in Y \text{\ and }(\id \ovot \chi)(z)\in X \text{\ for\ all } \omega \in B(\H)_*, \chi\in B(\K)_*\}$$
which is again an operator system. The Fubini tensor product obeys obvious commutativity, associativity and distributivity (with respect to $\ell^\infty$-direct sums of operator systems) properties. If either $X$ or $Y$ is finite-dimensional, then $X\ovot Y$ coincides with the algebraic tensor product $X\odot Y$. If $X,Y$ are both von Neumann algebras, then the Fubini tensor product is simply the von Neumann tensor product.

If $X_i\subseteq B(\H_i), i=1,2$ are operator systems and $\phi: X_1\to X_2$ is a ucp map, there is a unique ucp map 
$\phi \ovot \id: X_1\ovot B(\K)\to X_2\ovot B(\K)$
that extends $\phi\odot \id: X_1 \odot B(\K)\to X_2 \odot B(\K)$. Similarly, if $Y_i\subseteq B(\K_i), i=1,2$ are operator systems and $\psi: Y_1\to Y_2$ is ucp, we can define the ucp map $\id\ovot \psi: B(\H)\ovot Y_1\to B(\H)\ovot Y_2$. 

To make sure that everything works well, we need to impose continuity assumptions. This is illustrated by the following results, which will be used all the time without further mention.

\begin{lma}\label{Fubinitensor}\cite[I, Lemma 3.8]{Ham3}
    With notations as above and $Y\subseteq B(\K)$ an operator system, we have $(\phi \ovot \id)(X_1 \ovot Y)\subseteq X_2 \ovot Y$ if either $Y$ is $\sigma$-weakly closed or if $X_1$ and $X_2$ are $\sigma$-weakly closed and $\phi: X_1\to X_2$ is $\sigma$-weakly continuous. A similar result holds for $\id \ovot \psi$.
\end{lma}
\begin{prop}\cite[I, Lemma 3.9]{Ham3}
    With notations as above, if $Y_1$ and $Y_2$ are $\sigma$-weakly closed and $\psi: Y_1\to Y_2$ is $\sigma$-weakly continuous, then the compositions
$$(\id_{B(\H_2)}\ovot \psi)\circ (\phi \ovot \id_{B(\K_1)}), \quad(\phi \ovot \id_{B(\K_1)})\circ (\id_{B(\H_1)}\ovot \psi)$$
agree on $X_1\ovot Y_1$ and map $X_1\ovot Y_1$ into $X_2\ovot Y_2$. In this way, we obtain a canonical map $\phi\ovot \psi: X_1\ovot Y_1\to X_2\ovot Y_2$. If $\phi, \psi$ are complete order isomorphisms, then so is $\phi\ovot \psi$. Thus, as an abstract operator system, the Fubini tensor product only depends on the isomorphism classes of the factors.
\end{prop}

\begin{prop}
    Let $X\subseteq B(\H),Y\subseteq B(\K)$ be operator systems with $Y$ $\sigma$-weakly closed in $B(\K)$. Assume that $f\in X^*$ and $g\in Y_*$. There exist canonical slice maps
    $f\ovot \id_Y: X \ovot Y \to Y$ and $\quad \id_X \ovot g: X \ovot Y \to X$ that extend $f\odot \id: X \odot Y \to Y$ and $\id\odot g: X \odot Y \to X$, are norm-bounded, and satisfy
    $$g\circ (f\ovot \id_Y) = f \circ (\id_X \ovot g).$$
    This allows us to define the product functional $f\ovot g: X \ovot Y \to \mathbb{C}$. 
\end{prop}

Everything then works as expected. For example:
\begin{prop}\cite[Proposition 1.2]{Ham1} Assume $X_i\subseteq B(\H_i)$ and $Y_i\subseteq B(\K_i)$ are operator systems.
    Assume $\varphi: X_1\to X_2$ and $\psi: Y_1\to Y_2$ are ucp maps with $\psi$ $\sigma$-weakly continuous and $Y_i$ $\sigma$-weakly closed in $B(\K_i)$. If $f\in X_2^*$ and $g\in (Y_2)_*$, then
    \begin{align*}
        (f \ovot \id_{Y_2})\circ (\varphi \ovot \psi) = \psi \circ ((f\circ \varphi) \ovot \id_{Y_1}), \quad  (\id_{X_2}\ovot g)\circ (\varphi \ovot \psi) = \varphi \circ (\id_{X_1}\ovot (g\circ \psi)).
    \end{align*}
\end{prop}

We also mention the following result:

\begin{prop}\cite[I. Proposition 3.10]{Ham3} If $X\subseteq B(\H)$ is an operator system and $Y\subseteq B(\K)$ is a $\sigma$-weakly closed operator system, then $X \ovot Y$ is injective as an operator system if and only if both $X$ and $Y$ are injective as operator systems.
\end{prop}

If one encounters a Fubini tensor product (possibly with more than two factors) in this paper, it will always be the case that all factors, except possibly the first one, are von Neumann algebras and if tensor product maps between Fubini tensor products are considered, then there will always be $\sigma$-weakly continuous maps acting on these factors. Hence, we do not have to worry too much about technicalities that come with the Fubini tensor product in this paper.

\subsection{Compact and discrete quantum groups}

\textbf{Compact quantum groups.} \cite{Wor1, NT, Timmermann} A compact quantum group (= CQG) $\G$ consists of a pair $(C(\G), \Delta)$ where $C(\G)$ is a unital $C^*$-algebra and $\Delta: C(\G)\to C(\G)\otimes C(\G)$ a unital $*$-homomorphism satisfying coassociativity $(\id \otimes \Delta)\Delta = (\Delta \otimes \id)\Delta$ and  such that the quantum cancellation rules $[\Delta(C(\G))(1\otimes C(\G))]= C(\G)\otimes C(\G) = [\Delta(C(\G))(C(\G)\otimes 1)]$ hold.

Every CQG admits a unique state $\varphi_\G: C(\G)\to \mathbb{C}$, called Haar state, satisfying the following invariance conditions 
$$(\id \otimes \varphi_\G)\Delta(a) = \varphi_\G(a)1 = (\varphi_\G\otimes \id)\Delta(a), \quad a \in C(\G).$$
If the Haar state 
$\varphi_\G$ is a trace, $\G$ is said to be a CQG of Kac type. If the Haar state $\varphi_\G$ is faithful, then $\G$ is called reduced. 

With respect to the state $\varphi_\G$, we can consider the GNS-triplet $(L^2(\G), \lambda, \xi_\G)$ where $\lambda: C(\G)\to B(L^2(\G))$ is a $*$-representation and $\xi_\G\in L^2(\G)$ the cyclic vector such that $\varphi_\G= \langle \xi_\G, \lambda(-)\xi_\G\rangle$. We will also consider the GNS-map $\Lambda: C(\G)\to L^2(\G): a \mapsto \lambda(a)\xi_\G$ and we define $C_r(\G) = \lambda(C(\G))$.

A unitary finite-dimensional (= ufd) representation $\pi$ of $\G$ consists of a pair $(\H_\pi, U_\pi)$ where $\H_\pi$ is a finite-dimensional Hilbert space and $U_\pi \in B(\H_\pi)\otimes C(\G)$ a unitary such that $(\id \otimes \Delta)(U_\pi) = U_{\pi, 12}U_{\pi, 13}$ in $B(H)\otimes C(\G)\otimes C(\G)$. Given $\xi, \eta \in \H_\pi$, we define the matrix coefficient
$U_\pi(\xi, \eta):= (\omega_{\xi, \eta}\otimes \id)(U_\pi)$ where $\omega_{\xi, \eta}$ is the vector functional $x\mapsto \langle \xi, x\eta\rangle.$ We will write $n_\pi:= \dim(\H_\pi)$. If $\{e_i^\pi\}_{i=1}^{n_\pi}$ is an orthonormal basis for $\H_\pi$, then 
$$\Delta(U_\pi(\xi, \eta)) = \sum_{i=1}^{n_\pi} U_\pi(\xi, e_i^\pi)\otimes U_\pi(e_i^\pi, \eta).$$ 
If $\pi, \chi$ are ufd representations, then a linear map $T: \H_\pi \to \H_\chi$ is called intertwiner if $(T\otimes 1)U_\pi = U_\chi(T\otimes 1)$ and we write $\Mor(\pi, \chi)$ for the set of such intertwiners. A ufd representation $\pi$ of $\G$ is called irreducible if $\Mor(\pi, \pi)$ is one-dimensional. If $\pi, \chi$ are ufd representations, then $U_{\pi\boxtimes \chi}:= U_{\pi, 13} U_{\chi, 23}\in B(\H_\pi\otimes \H_\chi)\otimes C(\G)$ defines the tensor product representation $\pi\boxtimes \chi$. We will fix, once and for the remainder of this paper, a maximal family $\Irr(\G)$ of pairwise non-equivalent irreducible representations.
The space linearly spanned by all matrix coefficients is denoted by $\mathcal{O}(\G)$. It is norm-dense in $C(\G)$, $\varphi_\G$ is faithful on $\mathcal{O}(\G)$ and $(\mathcal{O}(\G), \Delta)$ forms a Hopf $^*$-algebra. A Hamel basis for $\mathcal{O}(\G)$ is constructed as follows: let $\{e_i^\pi\}_{i=1}^{n_\pi}$ be an orthonormal basis for $\H_\pi$ for every $\pi \in \Irr(\G)$. Then $\{U_\pi(e_i^\pi, e_j^\pi): \pi\in \Irr(\G), 1\le i,j\le n_\pi\}$ is a Hamel basis for $\mathcal{O}(\G)$. Counit and antipode act on matrix coefficients by
$$\epsilon(U_\pi(\xi, \eta)) = \langle \xi, \eta\rangle, \quad S(U_\pi(\xi, \eta)) = U_\pi(\eta, \xi)^*.$$ If $\pi\in \Irr(\G)$, there exists a canonical positive invertible operator $Q_\pi\in B(\H_\pi)$ such that $\Tr(Q_\pi^{1/2}) = \Tr(Q_\pi^{-1/2})$. We write $\dim_q(\pi)$ for the latter value and the following orthogonality relations hold:
\begin{align*}
    &\varphi_\G(U_\pi(\xi, \eta)U_\chi(\xi', \eta')^*) = \delta_{\pi, \chi}\frac{\langle \xi, \xi'\rangle \langle\eta', Q_\pi^{-1/2}\eta\rangle}{\dim_q(\pi)},\\
    &\varphi_\G(U_\pi(\xi, \eta)^*U_\chi(\xi', \eta')) = \delta_{\pi, \chi}\frac{\langle \xi', Q_\pi^{1/2}\xi\rangle \langle\eta, \eta'\rangle}{\dim_q(\pi)}.
\end{align*}
Given $\pi\in \Irr(\G)$, fix once and for all an orthonormal basis $\{e_i^\pi\}_{i=1}^n$ for $\H_\pi$ so that the operator $Q_\pi$ becomes diagonal, say $Q_\pi e_i^\pi = q_{\pi,i} e_i^\pi$. We then write $u_{ij}^\pi:= U_\pi(e_i^\pi, e_j^\pi)$ and the orthogonality relations take the form
$$\varphi_\G(u_{ij}^\pi (u_{kl}^{\chi})^*) = \frac{\delta_{\pi, \chi}\delta_{i,k}\delta_{j,l} q_{\pi,j}^{-1/2}}{\dim_q(\pi)}, \quad \varphi_\G((u_{ij}^\pi)^*u_{kl}^\chi)= \frac{\delta_{\pi, \chi} \delta_{i,k}\delta_{j,l}q_{\pi,i}^{1/2}}{\dim_q(\pi)}.$$
We also define the family of Woronowicz characters $\{\delta^z: \mathcal{O}(\G)\to \mathbb{C}\}_{z\in \mathbb{C}}$ uniquely determined by 
$$(\id \odot \delta^z)(U_\pi) = Q_\pi^{-z/2}, \quad \pi\in \Irr(\G).$$
With the CQG $\G$, we can associate the two multiplicative unitaries $V_\G, W_\G\in B(L^2(\G)\otimes L^2(\G))$ given by
$$V_\G(\Lambda(a)\otimes \Lambda(b)) = (\Lambda\odot \Lambda)(\Delta(a)(1\otimes b)), \quad W_\G^*(\Lambda(a)\otimes \Lambda(b)) = (\Lambda\odot \Lambda)(\Delta(b)(a\otimes 1))$$
where $a,b\in \mathcal{O}(\G)$. One has the relations
$$V_\G(\lambda(a)\otimes 1)V_\G^* = (\lambda\odot \lambda)(\Delta(a)) = W_\G^*(1\otimes \lambda(a))W_\G, \quad a \in \mathcal{O}(\G)$$ 
and the comultiplication $$\Delta: B(L^2(\G))\to B(L^2(\G))\ovot B(L^2(\G)): x \mapsto V_\G(x\otimes 1)V_\G^*$$
restricts to a comultiplication
$$\Delta: C_r(\G)\to C_r(\G)\otimes C_r(\G): x \mapsto V_\G(x\otimes 1)V_\G^* = W_\G^*(1\otimes x)W_\G.$$The pair $(C_r(\G), \Delta)$ then defines a reduced CQG $\G_r$ and should be thought of as a `reduced version of $\G$'. For most purposes, it suffices to work with the reduced CQG. We have $\mathcal{O}(\G_r)= \lambda(\mathcal{O}(\G))$. The counit $\epsilon_r$ for $\mathcal{O}(\G_r)$ satisfies $\epsilon_r\circ \lambda= \epsilon$ and similarly for the antipode $S_r$ and Haar state $\varphi_{\G_r}$. The von Neumann algebra $\mathscr{L}^\infty(\G):= \lambda(\mathcal{O}(\G))''$ carries the coproduct $$\Delta: \mathscr{L}^\infty(\G)\to \mathscr{L}^\infty(\G)\ovot \mathscr{L}^\infty(\G): x\mapsto V_\G(x\otimes 1)V_\G^* = W_\G^*(1\otimes x)W_\G$$
which extends the coproduct on $C_r(\G)$. The vector state $\omega_{\xi_\G}: \mathscr{L}^\infty(\G)\to \mathbb{C}: x \mapsto \langle \xi_\G, x \xi_\G\rangle$ is faithful and satisfies the invariance properties
$$(\omega_{\xi_\G}\ovot \id)\Delta= \omega_{\xi_\G}(-)1 = (\id \ovot \omega_{\xi_\G})\Delta$$
on $\mathscr{L}^\infty(\G)$.

\textbf{Discrete quantum groups.} We will introduce discrete quantum groups in the von Neumann algebra setting using duality theory for multiplier Hopf $*$-algebras \cite{KVD}. The pair $(\mathcal{O}(\G), \Delta)$ is a multiplier Hopf $^*$-algebra for which the duality theory in \cite{VD} applies. We therefore obtain a new dual multiplier Hopf $^*$-algebra $(c_c(\wG), \widetilde{\Delta})$ where $$c_c(\wG) =  \{\varphi_\G(-a): a \in \mathcal{O}(\G)\}\subseteq \mathcal{O}(\G)'= M(c_c(\wG))$$ and refer to $\wG$ as a discrete quantum group (= DQG). We define dual GNS maps $\hat{\Gamma}, \hat{\Lambda}: c_c(\wG)\to L^2(\G)$ by
$$\hat{\Gamma}(\varphi_\G(-a)) = \Lambda(a), \quad \hat{\Lambda}(\omega)= \hat{\Gamma}(\omega\delta^{-1/2}), \quad a \in \mathcal{O}(\G), \omega \in c_c(\wG).$$ We then use these to define the multiplicative unitaries $V_{\wG}, W_{\wG} \in B(L^2(\G)\otimes L^2(\G))$ by 
$$V_{\wG}(\hat{\Gamma}(\omega)\otimes \hat{\Gamma}(\chi)) = (\hat{\Gamma}\odot \hat{\Gamma})(\widetilde{\Delta}(\omega)(1\otimes \chi)), \quad W_{\wG}^*(\hat{\Lambda}(\omega)\otimes \hat{\Lambda}(\chi)) = (\hat{\Lambda}\odot \hat{\Lambda})(\widetilde{\Delta}(\chi)(\omega\otimes 1))$$
for $\omega, \chi \in c_c(\wG).$ We consider the left regular $*$-representation $\hat{\lambda}: c_c(\wG)\to B(L^2(\G))$ defined by $$\hat{\lambda}(\omega) \hat{\Gamma}(\chi)= \hat{\Gamma}(\omega\chi)$$ for $\omega, \chi \in c_c(\wG)$. Then we define the von Neumann algebra
$$\mathscr{L}^\infty(\wG):= \hat{\lambda}(c_c(\wG))''\subseteq B(L^2(\G)).$$

Next, we define the comultiplications
\begin{align*}&\hat{\Delta}_l: B(L^2(\G))\to B(L^2(\G))\ovot B(L^2(\G)): x \mapsto W_{\wG}^*(1\otimes x)W_{\wG}\\
&\hat{\Delta}_r: B(L^2(\G))\to B(L^2(\G))\ovot B(L^2(\G)): x \mapsto V_{\wG}(x\otimes 1)V_{\wG}^*\end{align*}
and we have $\hat{\Delta}:= \hat{\Delta}_l = \hat{\Delta}_r$ on $\mathscr{L}^\infty(\wG)$. There is a unique $*$-isomorphism
$$\mathscr{L}^\infty(\wG) \cong \prod_{\pi\in \Irr(\G)}^{\ell^\infty} B(\H_\pi)$$ that maps $\hat{\lambda}(\omega)$ to $((\id \odot \omega)(U_\pi))_{\pi\in \Irr(\G)}$ for every $\omega \in c_c(\wG)$. 
It has the property that
$$(p_\pi\ovot p_\chi)(\hat{\Delta}(x)) T = T p_\kappa(x)$$
for all $\pi, \chi, \kappa \in \Irr(\G), x \in \mathscr{L}^\infty(\wG)$ and $T \in \Mor(\kappa, \pi\boxtimes \chi)\subseteq B(\H_\kappa, \H_\pi\otimes \H_\chi)$, where $p_\pi: \mathscr{L}^\infty(\wG)\to B(\H_\pi)$ are the projections. We will also need the following facts:
$$V_\G \in \mathscr{L}^\infty(\wG)'\ovot \mathscr{L}^\infty(\G), \quad W_\G = V_{\wG}\in \mathscr{L}^\infty(\G)\ovot \mathscr{L}^\infty(\wG), \quad W_{\wG}\in \mathscr{L}^\infty(\wG)\ovot \mathscr{L}^\infty(\G)'.$$
The space $\hat{\lambda}(c_c(\wG))\lambda(\mathcal{O}(\G))$ is $\sigma$-weakly dense in $B(L^2(\G))$. 
If $g \in \mathcal{O}(\G)$, there is a unique normal functional
$\operatorname{ev}_g: \mathscr{L}^\infty(\wG)\to \mathbb{C}$ such that $\operatorname{ev}_g(\hat{\lambda}(\omega)) = \omega(g)$ for every $\omega \in c_c(\wG)$. Concretely, $\omega_{\xi, \eta}\circ p_\pi = \ev_{U_\pi(\xi, \eta)}$ for all $\xi, \eta \in \H_\pi$. We have $(\ev_g\ovot \ev_h)\circ \hat{\Delta}= \ev_{gh}$ for $g,h \in \mathcal{O}(\G)$ and $\hat{\epsilon}:= \ev_1 = \omega_{\xi_\G}$ is a character on $\mathscr{L}^\infty(\wG)$ satisfying 
$$(\id \ovot \hat{\epsilon})\hat{\Delta} = \id_{\mathscr{L}^\infty(\wG)} = (\hat{\epsilon}\ovot \id)\hat{\Delta}.$$ 

    If $\G$ is of Kac-type, then  on $B(L^2(\G))$ we have the identity
\begin{equation}\label{rem:Kacidentity}
(\id \ovot \omega_{\xi_\G})\hat{\Delta}_l = (\omega_{\xi_\G}\ovot \id)\hat{\Delta}_r.\end{equation}

\begin{ex}\label{example}
    \normalfont{Let $\Gamma$ be a discrete group with unit $e$ and consider the universal $C^*$-envelope $C^*(\Gamma)$ of the group $*$-algebra $\mathbb{C}[\Gamma]$. There is a unique $*$-morphism $\Delta: C^*(\Gamma)\to C^*(\Gamma)\otimes C^*(\Gamma)$ such that $\Delta(g) = g\otimes g$ for $g\in \Gamma$. The pair $\G = (C^*(\Gamma), \Delta)$ then defines a CQG. Its Haar state is given by $\varphi(g) = \delta_{g,e}$ for $g\in \Gamma$ and $\mathcal{O}(\G)=\mathbb{C}[\Gamma]$. We consider the Hilbert space $\ell^2(\Gamma)$ with canonical orthonormal basis $\{\delta_g\}_{g\in \Gamma}$. The left regular representation 
    $\lambda: C^*(\Gamma)\to B(\ell^2(\Gamma))$ together with the vector $\delta_e$ forms the GNS-construction with respect to the Haar state $\varphi$, and we thus see that $C_r(\G) = C_\lambda^*(\Gamma)$ and $\mathscr{L}^\infty(\G) = \mathscr{L}(\Gamma)$.

    It is straightforward to see that $c_c(\wG)= c_c(\Gamma)$, the finitely supported functions on $\Gamma$, with $\hat{\Gamma}(\varphi(-g)) = \delta_{g} = \hat{\Lambda}(\varphi(-g))$ for $g\in \Gamma$. Thus, $\hat{\lambda}(f)\delta_g = f(g^{-1})\delta_g$ for $g\in \Gamma$. Viewing $\ell^\infty(\Gamma)$ as multiplication operators on the Hilbert space $\ell^2(\Gamma)$, we have $\mathscr{L}^\infty(\wG) = \ell^\infty(\Gamma)$ and the coproduct $\hat{\Delta}: \ell^\infty(\Gamma)\to \ell^\infty(\Gamma)\ovot \ell^\infty(\Gamma)\cong \ell^\infty(\Gamma\times \Gamma)$ is given by $\hat{\Delta}(f)(g,h) = f(hg)$ for $f\in \ell^\infty(\Gamma)$ and $g,h \in \Gamma$.
    
    Finally, consider the orthogonal projection $e_g\in B(\ell^2(\Gamma))$ onto $\mathbb{C}\delta_g$ for $g\in \Gamma$. Then we have the following descriptions of the fundamental multiplicative unitaries
    $$V_\G = \sum_{g\in \Gamma} e_g \otimes \lambda_g, \quad W_\G = \sum_{g\in \Gamma} \lambda_g^*\otimes e_g, \quad W_{\wG}=\sum_{g\in \Gamma} e_g\otimes \rho_g^*$$
    where the sums converge in the strong operator topology. A word of caution: if $g\in \Gamma$, then $\ev_g$ as defined above for general discrete quantum groups does not agree with the usual definition of $\ev_g\in \ell^\infty(\Gamma)^*$. Rather, we have $\ev_g(f) = f(g^{-1})$ for $g\in \Gamma$ and $f\in \ell^\infty(\Gamma)$, which also explains the flip in the comultiplication.
    
    The CQG $\G$ discussed here is often denoted by $\widehat{\Gamma}$. 
    $\Di$}
\end{ex}

Throughout this entire paper, we will fix a CQG $\G$ with discrete dual $\wG$. If $X,Y$ are subsets of some normed algebra, then we write $XY= \lspan\{xy: x \in X, y \in Y\}$ and $[XY]$ is the norm-closure of $XY$.

\section{Actions of quantum groups}

In this section, we define actions of compact and discrete quantum groups on operator systems, establish their basic properties and show that these definitions generalise the definitions of actions known in the literature. Then, we define the reduced crossed product and the Fubini crossed product with respect to an operator system on which a DQG acts, and show that these crossed products admit actions of the dual CQG. These actions are then studied.

\subsection{Actions of compact quantum groups}

We will introduce two different notions of actions of a compact quantum group on an operator system. 
\begin{defn}\label{Gcstar}
    A right $\G$-$C^*$-operator system is a pair $(X, \alpha)$ where $X$ is an operator system and $\alpha: X \to X \otimes C_r(\G)$ is a uci map such that:
    \begin{itemize}
        \item The coaction property is satisfied, i.e. the diagram $$
\begin{tikzcd}
X \arrow[d, "\alpha"'] \arrow[rr, "\alpha"]          &  & X\otimes C_r(\G) \arrow[d, "\alpha \otimes \id"] \\
X\otimes C_r(\G) \arrow[rr, "\id \otimes \Delta"'] &  & X \otimes C_r(\G)\otimes C_r(\G)                
\end{tikzcd}$$
commutes. 
\item The Podleś density condition is satisfied, i.e. 
$[\alpha(X)(1\otimes C_r(\G))]= X \otimes C_r(\G).$
    \end{itemize}
    We will write $X \stackrel{\alpha}\curvearrowleft\G$ and say that $\alpha$ defines a right $\G$-$C^*$-action on the operator system $X$.
\end{defn}

We can define left actions in a similar way. In this paper, we will mainly consider right actions and all actions (also the $\G$-$W^*$-actions and $\wG$-actions which we will define later) are assumed to be right actions, unless otherwise specified. 

\begin{defn}\label{C*-action}
    Let $(X, \alpha)$ and $(Y, \beta)$ two $\G$-$C^*$-operator systems. A ucp map $\phi: X \to Y$ is called $\G$-$C^*$-equivariant if the diagram $$
\begin{tikzcd}
X \arrow[d, "\alpha"] \arrow[rr, "\phi"]        &  & Y \arrow[d, "\beta"] \\
X\otimes C_r(\G) \arrow[rr, "\phi \otimes \id"] &  & Y \otimes C_r(\G)   
\end{tikzcd}$$
commutes. When we want to emphasize the actions, we will denote the map by $\phi: (X, \alpha)\to (Y, \beta).$
\end{defn}

As can be expected, a $\G$-$C^*$-operator system admits a nice decomposition in spectral subspaces. In what follows, we deal with a setting that is more general than $\G$-$C^*$-operator systems. 

Given a ufd representation $\pi$ of $\G$, we will consider the associated right $\mathcal{O}(\G)$-comodule structure on $\H_\pi$ given by $\delta_\pi: \H_\pi \to \H_\pi \odot \mathcal{O}(\G): \xi\mapsto  U_\pi (\xi \otimes 1)$. Concretely, $\delta_\pi(e_i^\pi) = \sum_{j=1}^{n_\pi} e_j^\pi\otimes U_\pi(e_j^\pi, e_i^\pi)$ for every $1\le i \le n_\pi$.

\begin{defn}
    Suppose that $X$ is an operator system and $\alpha: X \to X \otimes C(\G)$ is a ccp map such that the coaction property and the Podleś density condition are satisfied. If  $\pi$ is a ufd representation of $\G$, the intertwiner space between $\pi$ and $\alpha$, denoted by $\Mor(\pi, \alpha)$, is the set of all linear maps $T: \H_\pi \to X$ such that the diagram
    $$
\begin{tikzcd}
\H_\pi \arrow[d, "\delta_\pi"'] \arrow[rr, "T"]            &  & X \arrow[d, "\alpha"] \\
\H_\pi\odot \mathcal{O}(\G) \arrow[rr, "T \odot \id"'] &  & X \otimes C(\G)    
\end{tikzcd}$$ commutes. We then define the $\pi$-spectral subspace 
$$X_\pi:= \lspan\{T\xi: T \in \Mor(\pi, \alpha), \xi \in \H_\pi\}.$$
\end{defn}
If $\tau\in \Irr(\G)$ is the trivial representation, then $X_\tau = \{x\in X: \alpha(x)= x \otimes 1\}= \operatorname{Fix}(X, \alpha)$.

The following result is well-known in the setting of actions of compact quantum groups on $C^*$-algebras. The proof contains no new elements, but we include it for the convenience of the reader.

\begin{prop}\label{algebraic core} Let $X$ be an operator system and $\alpha: X \to X \otimes C(\G)$ be a ccp map such that the coaction property and the  Podleś density condition are satisfied. Then $X_\pi$ is norm-closed in $X$ for every $\pi\in \Irr(\G)$. Moreover, $\mathcal{X}:= \sum_{\pi\in \Irr(\G)} X_\pi$ is norm-dense in $X$ and $$\mathcal{X}= \lspan\{(\id \otimes \varphi_\G)(\alpha(x)(1\otimes g)): x \in X, g \in \mathcal{O}(\G)\}.$$
The map $\alpha$ restricts to a Hopf $^*$-algebra coaction
$\alpha: \mathcal{X}\to \mathcal{X}\odot \mathcal{O}(\G).$ In particular, we have the identity
$\alpha(\mathcal{X})(1\otimes \mathcal{O}(\G))= \mathcal{X}\odot \mathcal{O}(\G).$ 

Moreover, if $\alpha$ is injective, then 
$\mathcal{X}=\{x\in X: \alpha(x)\in X\odot \mathcal{O}(\G)\}.$ 
\end{prop}
\begin{proof} If $\pi\in \Irr(\G)$, define $\chi_\pi:= \dim_q(\pi) \sum_{i=1}^{n_\pi} U_\pi(e_i^\pi, Q_\pi^{1/2}e_i^\pi)\in \mathcal{O}(\G)$ and consider the map
$$E_\pi: X\to X: x \mapsto (\id \otimes \varphi_{\G})(\alpha(x)(1\otimes \chi_\pi^*)).$$
If $x\in X$ and $\eta \in \H_\pi$, the map
$$T: \H_\pi \to X: \xi \mapsto (\id \otimes \varphi_\G)(\alpha(x)(1\otimes U_\pi(\xi, \eta)^*))$$ is an element of $\Mor(\pi, \alpha)$ which shows that $E_\pi(X)\subseteq X_\pi$. Combining this with the fact that $E_\pi(x) = x$ for $x\in X_\pi$, we conclude that $X_\pi=\{x\in X: E_\pi(x)= x\}$ is norm-closed. Clearly
$$\lspan \{(\id \otimes \varphi_\G)(\alpha(x)(1\otimes g)): x \in X, g \in \mathcal{O}(\G)\}=\mathcal{X}.$$
By the Podleś density condition, it follows that $\mathcal{X}$ is norm-dense in $X$. By definition, $\alpha(X_\pi)\subseteq X_\pi\odot \mathcal{O}(\G)_\pi$ so $\mathcal{X}\subseteq \{x\in X: \alpha(x)\in \mathcal{X} \odot \mathcal{O}(\G)\}$. It easily follows that $\alpha$ restricts to a Hopf $^*$-algebra coaction $\alpha: \mathcal{X}\to \mathcal{X}\odot \mathcal{O}(\G)$. 

Finally, assume that $\alpha$ is injective, and  $x\in X$ satisfies $\alpha(x)\in X \odot \mathcal{O}(\G)$. Write $\alpha(x)= \sum_{i=1}^n x_i\otimes g_i$. Choose $\omega \in c_c(\wG)\subseteq \mathcal{O}(\G)'$ such that $(\id \odot \omega)\Delta(g_i)= g_i$ for all $i=1, \dots, n$. Then
\begin{align*}
    \alpha((\id \odot \omega)\alpha(x)) = (\id \odot \id \odot \omega)(\alpha \odot \id)\alpha(x) = (\id \odot \id \odot \omega)(\id \odot \Delta)\alpha(x)=\alpha(x)
\end{align*}
and it follows that $x= (\id \odot \omega)\alpha(x) \in \mathcal{X}.$
\end{proof}

\begin{defn}
    The space $\mathcal{X}$ from Proposition \ref{algebraic core} is called the algebraic core of $\alpha$.
\end{defn}

It is natural to ask if we find back the usual notion of an action of a CQG on a $C^*$-algebra if we consider a $C^*$-algebra in Definition \ref{Gcstar}. This turns out to be true, but we need to prove a lemma first:

\begin{lma}\label{useful}
    Let $\H$ be a finite-dimensional Hilbert space and $V\in B(\H)\otimes C(\G)$ an invertible contraction such that $(\id \otimes \Delta)(V) = V_{12}V_{13}$. Then $V$ defines a ufd representation of $\G$.
\end{lma}
\begin{proof}
    Choose an invertible element $T\in \B(\H)$ such that $U:=(T\otimes 1)V(T^{-1}\otimes 1)$ is a unitary. Then $V= (T^{-1}\otimes 1)U(T\otimes 1)$ is contractive, and in particular $VV^* \le 1$, i.e.
     $$(T^{-1}\otimes 1)U(T\otimes 1)(T^*\otimes 1)U^*((T^*)^{-1}\otimes 1)\le 1$$
    so conjugating with $T$ gives 
    $U(TT^*\otimes 1)U^*\le TT^*\otimes 1.$
    Define the $*$-morphism $\beta_U: B(\H)\to B(\H)\otimes C(\G): x \mapsto U(x\otimes 1)U^*$ and define the faithful conditional expectation 
    $E: B(\H)\to B(\H)$ by $E:= (\id \otimes \varphi_\G)\circ \beta_U$. Let $\tau$ be a faithful state on $B(\H)$ (for example, normalised trace) and define $\omega:= \tau \circ E$. We calculate for $x\in B(\H)$, 
    \begin{align*}
        (\omega \otimes \id)\beta_U(x) &= (\tau \otimes \id)(\id \otimes \varphi_\G\otimes \id)(\beta_U\otimes \id)\beta_U(x)\\
        &= (\tau \otimes \id)(\id \otimes \varphi_\G\otimes \id)(\id \otimes \Delta)\beta_U(x)\\
        &= (\tau \otimes \id)((\id \otimes \varphi_\G)(\beta_U(x))\otimes 1)\\
        &= \tau(E(x))1 = \omega(x)1.
    \end{align*}
    In particular, we see that
    $0=(\omega \otimes \id)(TT^*\otimes 1-U(TT^*\otimes 1)U^*)$
    so by faithfulness of $\omega \otimes \id$, it follows that $TT^*\otimes 1 = U(TT^*\otimes 1)U^*$. Hence,
$$VV^* = (T^{-1}\otimes 1)U(TT^*\otimes 1)U^*((T^{-1})^*\otimes 1)=(T^{-1}\otimes 1)(TT^*\otimes 1)((T^{-1})^*\otimes 1)=1.$$
Since $V$ is invertible, it follows that $V$ is unitary.
\end{proof}
\begin{prop}
If $A$ is a $C^*$-algebra, $\G$ a reduced CQG and $\alpha: A \to A \otimes C(\G)$ is an injective ccp map satisfying the coaction property $(\id \otimes \Delta)\alpha = (\alpha\otimes \id)\alpha$ and the Podleś density condition, then $\alpha$ is multiplicative.  
\end{prop}

\begin{proof} Let $\omega$ be a positive invariant functional on $A$, i.e. $(\omega \otimes \id)(\alpha(a)) = \omega(a)1$ for all $a\in A$. Consider its associated GNS-Hilbert space $L^2_\omega(A)$ with GNS-vector $\xi_\omega$ and GNS-map $\Lambda_\omega: A \to L_\omega^2(A)$.

Consider the algebraic core $\mathcal{A}$ of the action $\alpha$.\footnote{The statement of Proposition \ref{algebraic core} is also true if the operator system is replaced by a (non-unital) $C^*$-algebra. The proof is exactly the same.} Let $a_1, \dots, a_n\in \mathcal{A}$ and $h_1, \dots, h_n \in \mathcal{O}(\G)$. Then using the fact that $[(\pi_\omega \otimes \lambda)(\alpha(a_i)^*\alpha(a_j))]_{i,j}\le [(\pi_\omega \otimes \lambda)(\alpha(a_i^*a_j))]_{i,j}$ \cite[Lemma 5.4]{Lance}, we find
\begin{align*}
    &\left\| \sum_i(\Lambda_\omega \odot \Lambda)( \alpha(a_i)(1\otimes h_i))\right\|^2\\
    &= \sum_{i,j}\langle (\Lambda_\omega \odot \Lambda)(\alpha(a_i)(1\otimes h_i)), (\Lambda_\omega\odot \Lambda)(\alpha(a_j)(1\otimes h_j)) \rangle\\
    &= \sum_{i,j}\langle \Lambda_\omega(1)\odot \Lambda(h_i), (\pi_\omega \otimes \lambda)(\alpha(a_i)^*\alpha(a_j))(\Lambda_\omega(1)\otimes \Lambda(h_j))\rangle\\
    &= \left\langle \begin{pmatrix}\Lambda_\omega(1)\otimes \Lambda(h_1) \\
    \vdots\\
    \Lambda_\omega(1)\otimes \Lambda(h_n)\end{pmatrix}, [(\pi_\omega \otimes \lambda)(\alpha(a_i)^*\alpha(a_j))]_{i,j}\begin{pmatrix}\Lambda_\omega(1)\otimes \Lambda(h_1) \\
    \vdots\\
    \Lambda_\omega(1)\otimes \Lambda(h_n)\end{pmatrix}\right\rangle\\
   &\le \left\langle \begin{pmatrix}\Lambda_\omega(1)\otimes \Lambda(h_1) \\
    \vdots\\
    \Lambda_\omega(1)\otimes \Lambda(h_n)\end{pmatrix}, [(\pi_\omega \otimes \lambda)(\alpha(a_i^*a_j))]_{i,j}\begin{pmatrix}\Lambda_\omega(1)\otimes \Lambda(h_1) \\
    \vdots\\
    \Lambda_\omega(1)\otimes \Lambda(h_n)\end{pmatrix}\right\rangle\\
    &= \sum_{i,j}\langle \Lambda_\omega(1)\otimes \Lambda(h_i), (\pi_\omega \otimes \lambda)(\alpha(a_i^*a_j))(\Lambda_\omega(1)\otimes \Lambda(h_j))\rangle\\
    &= \sum_{i,j} (\omega \otimes \varphi_\G)((1\otimes h_i^*)\alpha(a_i^*a_j)(1\otimes h_j))\\
    &= \sum_{i,j}\varphi_\G( h_i^* (\omega \otimes \id)(\alpha(a_i^*a_j))h_j)\\
    &= \sum_{i,j}\varphi_\G(h_i^*\omega(a_i^* a_j)h_j)\\
    &= \left\|\sum_{i} \Lambda_\omega(a_i)\otimes \Lambda(h_i)\right\|^2.
\end{align*}
Hence, there exists a unique contraction $V=V_{\omega}: L^2_\omega(A)\otimes L^2(\G)\to L^2_\omega(A)\otimes L^2(\G)$ such that $$V(\Lambda_\omega(a)\otimes \Lambda(h))= (\Lambda_\omega \odot \Lambda)(\alpha(a)(1\otimes h))$$ for $a\in \mathcal{A}$ and $h\in \mathcal{O}(\G)$. Let $C$ be any finite-dimensional linear subspace of $\mathcal{A}$ such that $\alpha(C)\subseteq C \odot \mathcal{O}(\G)$. Since $\alpha$ is a coaction and $V$ is well-defined, we can well-define the corepresentation
$$\delta_C: \Lambda_\omega(C)\to \Lambda_\omega(C)\odot \mathcal{O}(\G): \Lambda_\omega(a)\mapsto (\Lambda_\omega \odot \id)(\alpha(a)).$$
Then consider the linear map $$V_C: \Lambda_\omega(C) \odot \mathcal{O}(\G) \to \Lambda_\omega(C)\odot \mathcal{O}(\G): \Lambda_\omega(a)\otimes h \mapsto (\Lambda_\omega \odot \id)(\alpha(a)(1\otimes h)) = \delta_C(\Lambda_\omega(a))(1\otimes h).$$
Then $V_C$ is in the image of the linear embedding $B(\Lambda_\omega(C))\odot \mathcal{O}(\G)\hookrightarrow \operatorname{End}_{\mathbb{C}}(\Lambda_\omega(C)\odot \mathcal{O}(\G))$ and thus $V_C$ can be viewed as an invertible element inside $B(\Lambda_\omega(C))\otimes C(\G)\subseteq B(\Lambda_\omega(C)\otimes L^2(\G))$ satisfying $(\id \otimes \Delta)(V_C) = V_{C,12}V_{C,13}$. Moreover, $V_C$ then agrees with the contraction $V$ on the subspace $\Lambda_\omega(C)\odot \Lambda(\mathcal{O}(\G))$, so that $V_C$ is contractive. By Lemma \ref{useful}, $V_C$ is a unitary element of the $C^*$-algebra $B(\Lambda_\omega(C))\otimes C(\G)$. It follows therefore that $V$ is isometric on the subspace $\Lambda_\omega(C)\odot \Lambda(\mathcal{O}(\G))$. Since $C$ is an arbitrary finite-dimensional subcomodule of $\mathcal{A}$, it follows that $V$ is isometric. In particular, we thus see that if $a\in \mathcal{A}$, then
\begin{align*}
    (\omega \otimes \varphi_\G)(\alpha(a)^*\alpha(a))= \|(\Lambda_\omega \odot \Lambda)(\alpha(a))\|^2 = \|\Lambda_\omega(a)\otimes \Lambda(1)\|^2= (\omega \otimes \varphi_\G)(\alpha(a^*a)).
\end{align*}
Using faithfulness of $\varphi_\G$, we find
$$(\omega \otimes \id)\alpha(a^*a) = (\omega \otimes \id)(\alpha(a)^*\alpha(a)), \quad a \in \mathcal{A}.$$
This holds for all positive invariant functionals $\omega$ on $A$, in particular it also holds for the functionals $\sigma \circ E$ where $\sigma$ is a positive functional on $A$ and $E: A\to A: a \mapsto (\id \otimes \varphi_\G)(\alpha(a))$. Therefore, we get
$$(E \otimes \id)(\alpha(a^*a)) = (E\otimes \id)(\alpha(a)^*\alpha(a)).$$ 
Since $E\otimes \id$ is faithful on $A\otimes C(\G)$ (by injectivity of $\alpha$), we conclude that
$$\alpha(a^*a)= \alpha(a)^*\alpha(a)$$
for all $a\in \mathcal{A}$. By continuity of $\alpha$, this also holds for all $a\in A$. It then follows that the multiplicative domain of $\alpha$ is $A$, as desired.\end{proof}

We now introduce a second notion of action of a CQG on an operator system. The tensor product that occurs in the following definition is the Fubini tensor product, which was discussed in the preliminaries of this paper.

\begin{defn}\label{GWstar}
   A right $\G$-$W^*$-operator system is a pair $(X, \alpha)$ where $X$ is an operator system and $\alpha: X \to X \ovot \mathscr{L}^\infty(\G)$ is a uci map such that the coaction property is satisfied, i.e. the diagram $$
\begin{tikzcd}
X \arrow[d, "\alpha"'] \arrow[rr, "\alpha"]                      &  & X\ovot \mathscr{L}^\infty(\G) \arrow[d, "\alpha \ovot \id"]  \\
X\ovot \mathscr{L}^\infty(\G) \arrow[rr, "\id \ovot \Delta"'] &  & X \ovot \mathscr{L}^\infty(\G) \ovot \mathscr{L}^\infty(\G)
\end{tikzcd}$$
commutes.
\end{defn}
The notion of $\G$-$W^*$-equivariant map is defined in the obvious way: in Definition \ref{C*-action}, we just have to replace $C_r(\G)$ by $\mathscr{L}^\infty(\G)$ and the spatial tensor product by the Fubini tensor product.

\begin{defn}
If $(X, \alpha)$ is a $\G$-$W^*$-operator system, we define the algebraic regular elements and the regular elements by
$$\mathcal{R}_{\operatorname{alg}}(X, \alpha):= \{x\in X: \alpha(x)\in X \odot \mathcal{O}(\G_r)\},\quad \quad \mathcal{R}(X, \alpha):= \overline{\mathcal{R}_{\operatorname{alg}}(X, \alpha)}^{\|\cdot\|}.$$
\end{defn}
If $(X, \alpha)$ is a $\G$-$W^*$-operator system, it is easy to see that $\alpha(\mathcal{R}(X, \alpha))\subseteq \mathcal{R}(X, \alpha)\otimes C_r(\G)$ and that we obtain a $\G$-$C^*$-operator system for the restriction of the action. The Podleś condition follows because we have the identity
$$\mathcal{R}_{\operatorname{alg}}(X, \alpha)\odot \mathcal{O}(\G_r) = \alpha(\mathcal{R}_{\operatorname{alg}}(X, \alpha))(1\otimes \mathcal{O}(\G_r))$$
on the algebraic level. Conversely, if $(X, \alpha)$ is a $\G$-$C^*$-operator system, we can view it as a $\G$-$W^*$-operator system because $X\otimes C_r(\G)\subseteq X \ovot \mathscr{L}^\infty(\G)$.

Understanding one particular $\G$-$W^*$-action will be crucial for this work:

\begin{prop}\label{important action II} Let $X\subseteq B(\H)$ be an operator system and consider the map $\Delta: B(L^2(\G))\to B(L^2(\G))\ovot \mathscr{L}^\infty(\G): x \mapsto V_\G(x\otimes 1)V_\G^*$. Then $\id_X \ovot \Delta$ defines a right $\G$-$W^*$-action on $X \ovot B(L^2(\G))$ such that \begin{align*}&\mathcal{R}_{\operatorname{alg}}(X \ovot B(L^2(\G)), \id_X\ovot \Delta) = (X \ovot \mathscr{L}^\infty(\wG))(1\otimes \lambda(\mathcal{O}(\G))),\\ &\operatorname{Fix}(X\ovot B(L^2(\G)), \id_X\ovot \Delta)=X\ovot \mathscr{L}^\infty(\wG).\end{align*}
\end{prop}
\begin{proof} That $\id \ovot \Delta$ defines a $\G$-$W^*$-action is obvious, as is the assertion about the fixed points. We therefore focus on the description of the algebraic regular elements. For convenience of notation, we identify $\mathcal{O}(\G) \subseteq B(L^2(\G))$.

We first assume $X= B(\H)$. Suppose that $z\in B(\H)\ovot B(L^2(\G))$ satisfies $$(\id \ovot \Delta)(z)= \sum_{\pi\in \Irr(\G)}\sum_{i,j=1}^{n_\pi} z_{ij}^\pi \otimes u_{ij}^\pi$$ where $z_{ij}^\pi\in B(\H)\ovot B(L^2(\G))$. Since $B(\H)\odot \mathscr{L}^\infty(\wG)\mathcal{O}(\G)= (B(\H)\odot \mathscr{L}^\infty(\wG))(1\otimes \mathcal{O}(\G))$ is $\sigma$-weakly dense in $B(\H)\ovot B(L^2(\G))$, we can write
    $$z= \lim_\lambda \sum_{\pi \in \Irr(\G)}\sum_{i,j=1}^{n_\pi} z_{ij}^\pi(\lambda)(1\otimes u_{ij}^\pi)$$
    where $z_{ij}^\pi(\lambda)\in B(\H)\odot \mathscr{L}^\infty(\wG)$ with the limit taken in the $\sigma$-weak topology. Then
    \begin{align*}
       (\id \ovot \Delta)(z)= \lim_\lambda \sum_{\pi \in \Irr(\G)}\sum_{i,j,k=1}^{n_\pi} z_{ij}^\pi(\lambda)(1\otimes u_{ik}^\pi) \otimes u_{kj}^\pi.
    \end{align*}
    Applying the slice map $\id \ovot \id \ovot \omega_{\xi_\G}(-(u_{i'j'}^{\pi'})^*)$ to this expression and using the orthogonality relations, we get
    \begin{align*}
        z_{i'j'}^{\pi'} \frac{q_{\pi', j'}^{-1/2}}{\dim_q(\pi')}= \lim_\lambda \sum_i z_{ij'}^{\pi'}(\lambda)(1\otimes u_{ii'}^{\pi'})\frac{q_{\pi',j'}^{-1/2}}{\dim_q(\pi')}
    \end{align*}
    or equivalently
    $$ z_{i'j'}^{\pi'} = \lim_\lambda \sum_i z_{ij'}^{\pi'}(\lambda)(1\otimes u_{ii'}^{\pi'}).$$
    Multiply on the right with $1\otimes (u_{k'i'}^{\pi'})^*$ to obtain
    $$z_{i'j'}^{\pi'} (1\otimes (u_{k'i'}^{\pi'})^*)= \lim_\lambda \sum_i z_{ij'}^{\pi'}(\lambda)(1 \otimes u_{ii'}^{\pi'}(u_{k'i'}^{\pi'})^*).$$
    Summing over $i'$ gives
    $$\sum_{i'} z_{i'j'}^{\pi'}(1\otimes (u_{k'i'}^{\pi'})^*) = \lim_\lambda z_{k'j'}^{\pi'}(\lambda).$$
    In particular, the $\sigma$-weak limit $\lim_\lambda z_{i'j'}^{\pi'}(\lambda)\in B(\H)\ovot \mathscr{L}^\infty(\wG)$ exists for all $\pi'\in \Irr(\G)$ and $1\le i',j' \le n_{\pi'}$ and  thus
    $$z_{i'j'}^{\pi'}= \sum_i (\lim_\lambda z_{ij'}^{\pi'}(\lambda))(1\otimes u_{ii'}^{\pi'}) \in (B(\H)\ovot \mathscr{L}^\infty(\wG))(1\otimes \mathcal{O}(\G)).$$
    Consequently, 
    $$z= (\id_{B(\H\otimes L^2(\G))} \odot \epsilon_r)(\id \ovot \Delta)(z) = \sum_{\pi\in \Irr(\G)}\sum_{i=1}^{n_\pi} z_{ii}^\pi \in (B(\H)\ovot \mathscr{L}^\infty(\wG))(1\otimes \mathcal{O}(\G))$$
    as desired.
    
We now prove the general case, so let $X\subseteq B(\H)$ be any operator system.
If $z\in \mathcal{R}_{\operatorname{alg}}(X \ovot B(L^2(\G)))$, then also $z\in \mathcal{R}_{\operatorname{alg}}(B(\H) \ovot B(L^2(\G)))$. By the previous step, we can write 
$$z= \sum_{\pi\in \Irr(\G)}\sum_{i,j=1}^{n_\pi} z_{ij}^\pi(1\otimes u_{ij}^\pi)$$
where $z_{ij}^\pi\in B(\H)\ovot \mathscr{L}^\infty(\wG)$.
It follows that
$$(\id \ovot \Delta)(z) = \sum_{\pi \in \Irr(\G)}\sum_{i,j,k=1}^{n_\pi}z_{ij}^\pi(1\otimes u_{ik}^\pi)\otimes u_{kj}^\pi$$
so by another application of the orthogonality relations,
\begin{align*}
    (\id \ovot \id \odot \omega_{\xi_\G}(-(u_{k'j'}^{\pi'})^*))(\id \ovot \Delta)(z)&=\sum_{\pi \in \Irr(\G)} \sum_{i,j,k=1}^{n_\pi} z_{ij}^\pi(1\otimes u_{ik}^\pi)\varphi_\G(u_{kj}^\pi(u_{k'j'}^{\pi'})^*)\\
    &= \sum_{i=1}^{n_{\pi'}} \frac{q_{\pi',j'}^{-1/2}}{\dim_q(\pi')} z_{ij'}^{\pi'}(1\otimes u_{ik'}^{\pi'}).
\end{align*} This expression lives in $X \ovot B(L^2(\G))$, so that $\sum_{i=1}^{n_{\pi'}} z_{ij'}^{\pi'}(1\otimes u_{ik'}^{\pi'})\in X \ovot B(L^2(\G))$. Therefore, 
$$X \ovot B(L^2(\G))\ni \sum_{k'=1}^{n_{\pi'}}\left(\sum_{i=1}^{n_{\pi'}}  z_{ij'}^{\pi'}(1\otimes u_{ik'}^{\pi'})\right)(1\otimes (u_{i'k'}^{\pi'})^*)= z_{i'j'}^{\pi'}$$
 so we conclude that $z\in (X\ovot \mathscr{L}^\infty(\wG))(1\otimes \mathcal{O}(\G))$.
\end{proof}
    
Finally, we say something about injectivity of $\G$-equivariant operator systems.

\begin{defn}\label{Ginjectivity}
    A $\G$-$C^*$-operator system $I$ is called $\G$-$C^*$-injective if for all $\G$-$C^*$-operator systems $X,Y$, every $\G$-$C^*$-ucp map $\varphi: X \to I$ and every $\G$-$C^*$-uci map $\iota: X\to Y$, there exists a $\G$-$C^*$-ucp map $\Phi: Y \to I$ such that $\Phi \iota = \phi$. Similarly, $\G$-$W^*$-injectivity is defined.\end{defn}

    The following lemma was already established in \cite[Lemma 2.10]{HHN} in a slightly different context. The proof is unchanged, but we spell out the details for the convenience of the reader.

\begin{lma}\label{G-injective systems}
    Let $X$ be an injective operator system. Then
   $(X \ovot B(L^2(\G)), \id \ovot \Delta)$ is $\G$-$W^*$-injective.
\end{lma}
\begin{proof} Let $Y,Z$ be $\G$-$W^*$-operator systems $i: Y \to Z$ be a $\G$-$W^*$-uci map and $\phi: Y \to X \ovot B(L^2(\G))$ a $\G$-$W^*$-ucp map. Since $X \ovot B(L^2(\G))$ is an injective operator system, there exists a ucp map $\Theta: Z \to X\ovot B(L^2(\G))$ such that $\Theta i = \phi$. 
    
    Consider the coaction $\alpha_Z: Z \to Z \ovot \mathscr{L}^\infty(\G)$ and the Haar state $\varphi_\G = \omega_{\xi_\G}: \mathscr{L}^\infty(\G)\to \mathbb{C}$ and define 
    $$\Phi: Z \to X \ovot B(L^2(\G)): z \mapsto (\id \ovot \id \ovot \varphi_\G)(V_{\G, 23}^*(\Theta \ovot \id)(\alpha_Z(z))V_{\G, 23}).$$
    Then for $y\in Y$,
    \begin{align*}
        \Phi i(y) &= (\id \ovot \id \ovot \varphi_\G)(V_{\G, 23}^*(\Theta \ovot \id)(\alpha_Z(i(y))V_{\G, 23}))\\
        &= (\id \ovot \id \ovot \varphi_\G)(V_{\G, 23}^*(\Theta \ovot \id)(i \ovot \id)(\alpha_Y(y))V_{\G, 23})\\
        &=(\id \ovot \id \ovot \varphi_\G)(V_{\G, 23}^*(\phi \ovot \id)(\alpha_Y(y))V_{\G, 23})\\
        &= (\id \ovot \id \ovot \varphi_\G)(V_{\G, 23}^* (\id \ovot \Delta)(\phi(y))V_{\G, 23})\\
        &= (\id \ovot \id \ovot \varphi_\G)(\phi(y)\otimes 1)= \phi(y)
    \end{align*} so $\Phi \circ i = \phi$. Moreover, $\Phi$ is $\G$-$W^*$-equivariant. To see this, recall that $(\id \ovot \Delta)(V_\G)= V_{\G, 12}V_{\G, 13}$, so that $V_\G\otimes 1= (\id \ovot \Delta)(V_\G)V_{\G, 13}^*$. Therefore, if $z\in Z$,
\begin{align*} &(\Phi \ovot \id)(\alpha_Z(z))\\
&= (\id \ovot \id \ovot \varphi_\G\ovot \id)((1\otimes V_{\G}^*\otimes 1)(\Theta \ovot \id \ovot \id)((\alpha_Z \ovot \id)\alpha_Z(z))(1\otimes V_{\G}\otimes 1))\\
&= (\id \ovot \id \ovot \varphi_\G\ovot \id)((1\otimes V_{\G}^*\otimes 1)(\Theta \ovot \id \ovot \id)((\id \ovot \Delta)\alpha_Z(z))(1\otimes V_{\G}\otimes 1))\\
&= (1\otimes V_{\G})(\id \ovot \id \ovot \varphi_\G\ovot \id)(\id \ovot \id \ovot \Delta)((1\otimes V_\G^*)(\Theta \ovot \id)(\alpha_Z(z))(1\otimes V_\G)))(1\otimes V_{\G}^*)\\
&= (1\otimes V_\G)(\id \ovot \id \ovot \varphi_\G(-)1)((1\otimes V_\G^*)(\Theta \ovot \id)(\alpha_Z(z))(1\otimes V_\G))(1\otimes V_\G^*)\\
&= (1 \otimes V_\G)(\Phi(z)\otimes 1)(1\otimes V_\G^*)\\
&= (\id \ovot \Delta)(\Phi(z)).
\end{align*}
and the proof is finished.
\end{proof}

\begin{rem} \normalfont{If $X$ is an injective operator system, it is in general not true that the $\G$-$W^*$-operator system $(X\ovot \mathscr{L}^\infty(\G), \id\ovot \Delta)$ is $\G$-$W^*$-injective. Indeed, this already fails for $X=\mathbb{C}$: we will later see that $\mathscr{L}^\infty(\G)$ is $\G$-$W^*$-injective if and only if $\G$ is coamenable (see Corollary \ref{amenability}). However, on the dual side, this result is true (see Lemma \ref{wG-injectivity}). $\Di$}
\end{rem}
\begin{prop}\label{G-injectivity characterisation}
    Let $(X, \alpha)$ be a $\G$-$W^*$-operator system. The following statements are equivalent: 
\begin{enumerate}
    \item $(X, \alpha)$ is $\G$-$W^*$-injective.
    \item $X$ is injective and there exists a $\G$-$W^*$ ucp conditional expectation $\varphi: (X \ovot B(L^2(\G)), \id \ovot \Delta)\to (\alpha(X), \id \ovot \Delta)$.
\end{enumerate}
\end{prop}
\begin{proof}
    $(2)\implies (1)$ By Lemma \ref{G-injective systems}, $(X\ovot B(L^2(\G)), \id \ovot \Delta)$ is $\G$-$W^*$-injective. Thus $(1)$ immediately follows.

    $(1)\implies (2)$ Suppose $X \subseteq B(\H)$. Since $X$ is $\G$-$W^*$-injective, there is an equivariant ucp map $\theta: (B(\H\otimes L^2(\G)), \id_{B(\H)} \ovot \Delta)\to (X, \alpha)$ such that the diagram
    $$
\begin{tikzcd}
X \arrow[rr, "\alpha"] \arrow[d, "\id_X"'] &  & X\ovot \mathscr{L}^\infty(\G) \arrow[d, "\subseteq"] \\
X                                          &  & B(\H\otimes L^2(\G)) \arrow[ll, "\theta"]
\end{tikzcd}$$
commutes. Since $B(\H\otimes L^2(\G))$ is injective and $\theta \alpha = \id_X$, it follows that $X$ is injective.
\end{proof}

\begin{rem}
    \normalfont{In the previous proposition, we have seen that $\G$-$W^*$-injectivity implies injectivity of the operator system. For $\G$-$C^*$-injectivity, this is no longer true. See Remark \ref{bad}.$\Di$}
\end{rem}

We end this subsection with some trivial, yet useful observations.
\begin{lma}\label{regular injectivity}
    If $(X, \alpha)$ is a $\G$-$W^*$-injective operator system, then $(\mathcal{R}(X), \alpha)$ is a $\G$-$C^*$-injective operator system.
\end{lma}
\begin{proof}
    Let $Y,Z$ be $\G$-$C^*$-operator systems, $i: Y \to Z$ be a $\G$-$C^*$-uci map and $\phi: Y \to \mathcal{R}(X)$ a $\G$-$C^*$- ucp map. Viewing $Y,Z$ as $\G$-$W^*$-operator systems, we can apply $\G$-$W^*$-injectivity of $(X, \alpha)$ to find a $\G$-$W^*$-ucp map $\Phi: Z \to X$ such that $\Phi\circ i = \phi$. But clearly $\Phi$ maps $Z= \mathcal{R}(Z)$ into $\mathcal{R}(X)$ and viewing $\Phi$ as a map $Z\to \mathcal{R}(X)$, it is $\G$-$C^*$-equivariant.
\end{proof}

\begin{cor} If $X$ is an injective operator system, then
    $\mathcal{R}(X \ovot B(L^2(\G)), \id_X \ovot \Delta)$ is $\G$-$C^*$-injective.
\end{cor}

  In the literature, the notion of equivariant injectivity has been studied in the category with equivariant unital $C^*$-algebras as objects and equivariant ucp maps between such $C^*$-algebras as morphisms (and equivariant uci maps as monomorphisms), see for instance \cite{KKSV}. We make the somewhat trivial observation that this notion of equivariant injectivity is equivalent to the one that we defined.

\begin{prop}
    Let $(A, \alpha)$ be a unital $\G$-$C^*$-algebra. Then $(A, \alpha)$ is injective in the category of unital $\G$-$C^*$-algebras with ucp maps if and only if it is injective in the category of $\G$-$C^*$-operator systems. 
\end{prop}
\begin{proof}
    Assume that $(A, \alpha)$ is injective in the category of unital $\G$-$C^*$-algebras. We may assume that $A\subseteq B(\H)$ for some Hilbert space $\H$. There exists a $\G$-$C^*$-ucp map $\varphi: \mathcal{R}(B(\H)\ovot B(L^2(\G)))\to A$ such that the diagram
    $$
\begin{tikzcd}
A \arrow[rr, "\alpha"] \arrow[d, "\id_A"'] &  & {\mathcal{R}(B(\H)\ovot B(L^2(\G)), \id\ovot \Delta)} \arrow[lld, "\varphi"] \\
A                                          &  &                                                                              
\end{tikzcd}$$
commutes. Since $\mathcal{R}(B(\H)\ovot B(L^2(\G)), \id \ovot \Delta)$ is $\G$-$C^*$-injective in the category of $\G$-$C^*$-operator systems, it follows that $A$ is also $\G$-$C^*$-injective as an operator system. The converse is clear.
\end{proof}

A similar result obviously also holds for unital $\wG$-$C^*$-algebras.

\begin{lma}
   $(\mathbb{C}, \tau)$, where $\tau$ is the trivial action, is both $\G$-$C^*$ and $\G$-$W^*$-injective.
\end{lma}
\begin{proof}
    Immediate since the map $\omega_{\xi_\G}: B(L^2(\G)) \to \mathbb{C}$ is $(\Delta, \tau)$-equivariant. \end{proof}
\subsection{Actions of discrete quantum groups}

We now introduce the notion of an operator system on which a discrete quantum group acts.

\begin{defn}\label{action DQG} A (right) $\wG$-operator system is a pair $(X, \alpha)$ where $X$ is an operator system and $\alpha: X \to X \ovot \mathscr{L}^\infty(\wG)$ is a uci map such that the diagram $$
\begin{tikzcd}
X \arrow[rr, "\alpha"] \arrow[d, "\alpha"']                          &  & X \ovot \mathscr{L}^\infty(\wG) \arrow[d, "\alpha \ovot \id"] \\
X \ovot \mathscr{L}^\infty(\wG) \arrow[rr, "\id \ovot \hat{\Delta}"] &  & X \ovot \mathscr{L}^\infty(\wG)\ovot \mathscr{L}^\infty(\wG) 
\end{tikzcd}$$
commutes. We will write $X \stackrel{\alpha}\curvearrowleft\wG$ and say that $\alpha$ defines a right $\wG$-action on the operator system $X$.
\end{defn}
The notion of $\wG$-equivariant maps between $\wG$-operator systems is defined in the same way as for $\G$-$W^*$-operator systems, as is the notion of $\wG$-injective operator system. 

Recall also that $\wG$ is called amenable \cite{Tom}, or equivalently $\G$ is coamenable, if there exists a $\wG$-equivariant state $(\mathscr{L}^\infty(\wG), \hat{\Delta})\to (\mathbb{C}, \tau)$ where $\tau$ is the trivial action $\mathbb{C} \curvearrowleft \wG$. 

\begin{lma} If $X \stackrel{\alpha}\curvearrowleft\wG$ is a $\wG$-operator system, then $(\id \ovot \hat{\epsilon})\circ \alpha = \id$.
\end{lma}
\begin{proof}
    If $x\in X$, then
    \begin{align*}
        \alpha(\id \ovot \hat{\epsilon})\alpha(x) = (\id \ovot \id \ovot \hat{\epsilon})(\alpha \ovot \id)\alpha(x) = (\id \ovot \id \ovot \hat{\epsilon})(\id \ovot \hat{\Delta})\alpha(x) = \alpha(x)
    \end{align*}
    so by injectivity of $\alpha$ we conclude that $(\id \ovot \hat{\epsilon})\alpha(x) =x$. 
\end{proof}
If $A$ is a unital $C^*$-algebra, the Fubini tensor product $A\ovot \mathscr{L}^\infty(\wG) \cong \prod_{\pi \in \Irr(\G)} A \otimes B(\H_\pi)$ carries a natural $C^*$-algebra structure. It then makes sense to ask if an action of a DQG on a unital $C^*$-algebra in the sense of Definition \ref{action DQG} is automatically multiplicative. This turns out to be true, as the following result shows. The proof of this fact was communicated to us by Stefaan Vaes on MathOverflow \cite{MO}.

\begin{prop}\label{multiplicative} Assume that $A$ is a unital $C^*$-algebra and $\alpha: A \to A \ovot \mathscr{L}^\infty(\wG)$ is an injective ucp map satisfying the coaction property $(\alpha \ovot \id)\alpha = (\id \ovot \hat{\Delta})\alpha$. Then $\alpha$ is a $*$-homomorphism.
\end{prop}
\begin{proof}
    Consider for $\pi\in \Irr(\G)$ the ucp maps $\alpha_\pi:= (\id_A \ovot p_\pi)\circ \alpha: A \to A \odot B(\H_\pi)$. It suffices to show that these maps are multiplicative.
    
    If $\tau\in \Irr(\G)$ is the trivial representation, then $p_\tau = \hat{\epsilon}$ so that $\alpha_\tau = \id_A$. Choose $\rho \in \Irr(\G)$ and $s \in \Mor(\tau, \pi\boxtimes \rho), t \in \Mor(\tau, \rho \boxtimes \pi)$ such that $t^*t = 1$ and $(s^*\otimes 1)(1\otimes t) = 1$. Define the ucp map
    $$\theta_\pi: A \odot B(\H_\pi)\to A: x \mapsto (1\otimes t^*)(\alpha_\rho \odot \id)(x)(1\otimes t).$$
    We claim that $\theta_\pi$ is faithful. Indeed, let $x\in A \odot B(\H_\pi)$ with $\theta_\pi(x^*x)=0$. Then 
   $(\alpha_\rho \odot \id)(x)(1\otimes t)= 0$ by the Schwarz-inequality. It follows that
   \begin{align*}
       0 &= (1 \otimes s^*\otimes 1)(\alpha_\pi \ovot \id \ovot \id)((\alpha_\rho \ovot \id)(x))(1\otimes 1 \otimes t)\\
       &= (1\otimes s^*\otimes 1)(\id \ovot p_\pi \ovot p_\rho \ovot \id)(\alpha \ovot \id\ovot \id)(\alpha \ovot \id)(x)(1\otimes 1\otimes t)\\
       &= (1\otimes s^*\otimes 1)(\id \ovot p_\pi \ovot p_\rho \ovot \id)(\id \ovot \hat{\Delta} \ovot \id)(\alpha \ovot \id)(x)(1\otimes 1 \otimes t)\\
       &= (\id \ovot p_\tau \ovot \id)(\alpha \ovot \id)(x)(1\otimes s^*\otimes 1)(1\otimes 1 \otimes t)\\
       &= x(1\otimes s^*\otimes1)(1\otimes 1\otimes t) = x
   \end{align*}
   and faithfulness of $\theta_\pi$ is proven. For $a\in A$, we have $\theta_\pi(\alpha_\pi(a)) = a$ so if $u\in A$ is a unitary, it follows that $\alpha_\pi(u)$ is in the multiplicative domain of $\theta_\pi$. By the Schwarz inequality, $\alpha_\pi(u)^*\alpha_\pi(u)\le \alpha_\pi(u^*u) = 1$ and 
   $$\theta_\pi(1-\alpha_{\pi}(u)^*\alpha_\pi(u)) = 0$$
   so that faithfulness of $\theta_\pi$ ensures that $\alpha_\pi(u)^*\alpha_\pi(u)= 1$. It follows that $u$ is in the multiplicative domain of $\alpha_\pi$ for every unitary $u\in A$, and thus $\alpha_\pi$ is multiplicative.
\end{proof}

We now give another description of $\wG$-operator systems in terms of $\mathcal{O}(\G)$-modules:

\begin{defn}\label{module operator system} A left $\mathcal{O}(\G)$-module operator system consists of an operator system $X$ and a linear map $\rhd: \mathcal{O}(\G)\odot X \to X$ such that the following conditions are satisfied:
\begin{enumerate}
    \item $1 \rhd x = x$ for all $x\in X$.
    \item $g \rhd 1= \epsilon(g)1$ for all $g\in \mathcal{O}(\G)$.
    \item $g \rhd (h\rhd x)= (gh)\rhd x$ for all $g,h \in \mathcal{O}(\G)$ and all $x\in X$.
    \item The map $X \to M_{n_\pi}(X): x \mapsto [u_{ij}^\pi \rhd x]_{i,j}$ is ucp for every $\pi \in \Irr(\G).$
\end{enumerate}
\end{defn}

\begin{lma}
    If $(X, \alpha)$ is a $\wG$-operator system, then $g\rhd x:= (\id \ovot \ev_{g})\alpha(x)$ defines a left $\mathcal{O}(\G)$-module operator system.
\end{lma}
\begin{proof}
    The first three conditions in Definition \ref{module operator system} are obvious. The fourth condition is clear since the map $x\mapsto [u_{ij}^\pi \rhd x]_{i,j}$ is exactly the composition
    $$
\begin{tikzcd}
X  \arrow[rr, "\alpha"] &  & X \ovot \mathscr{L}^\infty(\wG) \arrow[rr, "\id \ovot p_\pi"] &  & X \ovot B(\H_\pi) \arrow[rr, "\cong"] &  & M_{n_\pi}(X) 
\end{tikzcd}$$
of ucp maps.
\end{proof}

The following result does not come unexpected:
\begin{prop}
    If $(X, \rhd)$ is a left $\mathcal{O}(\G)$-module operator system, there is a unique right action $X\stackrel{\alpha}\curvearrowleft \wG$ such that $g \rhd x = (\id \ovot \ev_g)\alpha(x)$ for all $g\in \mathcal{O}(\G)$ and all $x\in X$.
\end{prop}
\begin{proof} Let $\{E_{ij}^\pi\}_{i,j=1}^{n_\pi}$ be the matrix units w.r.t. the orthonormal basis $\{e_i^\pi\}_{i=1}^{n_\pi}$. Given $\pi \in \Irr(\G)$, define 
    $$\alpha_\pi: X \to X \odot B(\H_\pi): x \mapsto \sum_{i,j=1}^{n_\pi}(u_{ij}^\pi \rhd x) \otimes E_{ij}^\pi.$$
    Then $\alpha_\pi$ is a completely positive map, so that the induced map
    $$\alpha = \prod_{\pi \in \Irr(\G)} \alpha_\pi: X \to \prod_{\pi \in \Irr(\G)}(X \odot B(\H_\pi)) \cong X \ovot \mathscr{L}^\infty(\wG)$$ is ucp In particular, $\alpha$ is completely contractive. Clearly $g\rhd x = (\id \otimes \ev_g)\alpha(x)$ for all $g\in \mathcal{O}(\G)$ and all $x \in X$. In particular, note that $(\id \ovot \hat{\epsilon})\alpha = (\id \ovot \ev_1)\alpha = \id$ so that all the matrix amplifications of $\alpha$ are norm-increasing. Hence, $\alpha$ is a complete isometry. It is now clear that $\alpha$ defines a right action $X\curvearrowleft \wG$.
\end{proof}

Also the following is well-known in other contexts, see e.g. \cite[Proposition 2.5]{HHN}. We include the easy proof for convenience.

\begin{lma}[Poisson transformation]\label{lemmaPoissontransform}
    Given a $\wG$-operator system $(X,\alpha)$ and an operator system $Y$, the assignments 
    $$\phi \mapsto \mathcal{P}_\phi:= (\phi \ovot \id)\circ \alpha, \quad \quad\mathcal{P}\mapsto \phi_\mathcal{P}:=(\id \ovot \hat{\epsilon})\circ \mathcal{P}$$
    define a bijective correspondence between ucp maps $X\to Y$ and $\wG$-equivariant ucp maps $(X, \alpha)\to (Y\ovot \mathscr{L}^\infty(\wG), \id \ovot \hat{\Delta})$.
\end{lma}
\begin{proof} If $\phi: X \to Y$ is ucp, then $\mathcal{P}_\phi$ is $\wG$-equivariant since for $x\in X$,
\begin{align*}
    (\id \ovot \hat{\Delta})\mathcal{P}_\phi(x)&= (\id \ovot \hat{\Delta})(\phi \ovot \id)\alpha(x)\\&= (\phi \ovot \id \ovot \id)(\id \ovot \hat{\Delta})(\alpha(x))\\
    &= (\phi \ovot \id \ovot \id)(\alpha \ovot \id)(\alpha(x))\\
    &= (\mathcal{P}_\phi \ovot \id)\alpha(x).
\end{align*}
Next, note that for $x\in X$,
\begin{align*}&\phi_{\mathcal{P}_\phi}(x) = (\id \ovot \hat{\epsilon})(\phi \ovot \id)\alpha(x)= \phi (\id \ovot \hat{\epsilon})\alpha(x) = \phi(x)\\
&\mathcal{P}_{\phi_\mathcal{P}}(x)= (\id \ovot \hat{\epsilon}\ovot \id)(\mathcal{P}\ovot \id)\alpha(x)= (\id \ovot \hat{\epsilon} \ovot \id)(\id \ovot \hat{\Delta})\mathcal{P}(x)= \mathcal{P}(x)
\end{align*}
so that $\mathcal{P}\mapsto \phi_\mathcal{P}$ and $\phi\mapsto \mathcal{P}_\phi$ are inverse to each other.
\end{proof}

\begin{lma}\label{wG-injectivity}
    If $X$ is an injective operator system, then $(X \ovot \mathscr{L}^\infty(\wG), \id \ovot \hat{\Delta})$ is $\wG$-injective.
\end{lma}
\begin{proof}
    Let $Y,Z$ be $\wG$-operator systems, $\kappa: Y \to Z$ a unital $\wG$-equivariant complete isometry and $\varphi: Y \to X\ovot \mathscr{L}^\infty(\wG)$ a $\wG$-equivariant ucp map. By injectivity of $X$, there exists a ucp map $\theta: Z \to X$ such that the diagram
$$
\begin{tikzcd}
Y \arrow[d, "\kappa"'] \arrow[rr, "\varphi"] &  & X \ovot\mathscr{L}^\infty(\wG) \arrow[rr, "\id \ovot \hat{\epsilon}"] &  & X \\
Z \arrow[rrrru, "\theta"', dashed]   &  &                                                                  &  &  
\end{tikzcd}$$
commutes. Then, the Poisson transform $\Theta := \mathcal{P}_\theta = (\theta\ovot \id)\alpha_Z: Z \to X \ovot \mathscr{L}^\infty(\wG)$ is a ucp $\wG$-equivariant map. Moreover, we have
$$\Theta \kappa = (\theta\ovot \id)\alpha_Z \kappa = (\theta\ovot \id)(\kappa \ovot \id)\alpha_Y = (\id \ovot \hat{\epsilon}\ovot \id)(\varphi \ovot \id)\alpha_Y = (\id \ovot \hat{\epsilon}\ovot \id)(\id \ovot \hat{\Delta}) \varphi = \varphi \ ,$$
which ends the proof.
\end{proof}

\begin{prop}\label{hat injectivity}
    Let $(X, \alpha)$ be a $\wG$-operator system. The following statements are equivalent:
    \begin{enumerate}
        \item $(X, \alpha)$ is $\wG$-injective.
        \item $X$ is injective and there exists a $\wG$-equivariant ucp conditional expectation $(X\ovot \mathscr{L}^\infty(\wG), \id \ovot \hat{\Delta}) \to (\alpha(X), \id \ovot \hat{\Delta})$.
        \item $X$ is injective and there exists a $\wG$-equivariant ucp conditional expectation $(X\ovot B(L^2(\G)), \id \ovot \hat{\Delta}_r)\to (\alpha(X), \id \ovot \hat{\Delta})$.
    \end{enumerate}
\end{prop}
\begin{proof} $(1)\iff (2)$ The proof is the same as Proposition \ref{G-injectivity characterisation}.

$(3)\implies (2)$ Obvious.

$(1)\implies (3)$ From the implication $(1)\implies (2)$, we know that $X$ is injective, and the existence of a $\wG$-equivariant ucp conditional expectation $X\ovot B(L^2(\G))\to \alpha(X)$ is immediate from $\wG$-injectivity. Alternatively, we can also prove $(2)\implies (3)$ directly by noting that the map
$$(B(L^2(\G)), \hat{\Delta}_r)\to (\mathscr{L}^\infty(\wG), \hat{\Delta}): x \mapsto (\omega_{\xi_\G}\ovot \id)(\hat{\Delta}_r(x))$$ is a $\wG$-equivariant conditional expectation, which follows immediately from the coassociativity of $\hat{\Delta}_r$.
\end{proof}

We will now introduce the reduced and Fubini crossed product associated to a $\wG$-operator system. To ensure that the reduced crossed product is an operator system, we will need the following lemma:

\begin{lma}
    If $(X, \alpha)$ is a $\wG$-operator system, $x\in X$ and $g\in \mathcal{O}(\G)$, then
    $$(\alpha(x)(1\otimes \lambda(g))^*= \alpha((\id \ovot \ev_{g_{(1)}^*})(\alpha(x^*)))(1\otimes \lambda(g_{(2)}^*)).$$
\end{lma}
\begin{proof} We assume $X\subseteq B(\H)$.
    If $\omega \in c_c(\wG)$ and $h\in \mathcal{O}(\G)$, it is not hard to see that the commutation relation
    $$\hat{\lambda}(\omega(-h_{(1)}))\lambda(h_{(2)}) = \lambda(h)\hat{\lambda}(\omega)$$
    holds. Consequently, we see that
    \begin{align*}
        (\id \ovot \ev_{h_{(1)}})(\hat{\Delta}(z))\lambda(h_{(2)}) = \lambda(h)z
    \end{align*}
    for every $z\in \mathscr{L}^\infty(\wG)$ and every $h\in \mathcal{O}(\G)$. Thus, if $\phi\in B(\H)_*$, we calculate
    \begin{align*}
        &(\phi \ovot \id)[\alpha((\id \ovot \ev_{g_{(1)}^*})\alpha(x^*))(1\otimes \lambda(g_{(2)}^*))]\\
        &= (\phi \ovot \id)[(\id \ovot \id \ovot \ev_{g_{(1)}^*})((\id \ovot \hat{\Delta})(\alpha(x^*)))(1\otimes \lambda(g_{(2)}^*))]\\
        &= (\id \ovot \ev_{g_{(1)}^*})[\hat{\Delta}((\phi \ovot \id)(\alpha(x^*)))]\lambda(g_{(2)}^*))\\
        &= \lambda(g^*)(\phi \ovot \id)(\alpha(x^*))\\&= (\phi \ovot \id)((1\otimes \lambda(g^*))\alpha(x^*)).
    \end{align*}
    The lemma follows.
\end{proof}

\begin{defn}\label{reduced}
Let $(X, \alpha)$ be a $\wG$-operator system. We define the reduced crossed product to be the operator system $X\rtimes_{\alpha,r} \wG := \clspan^{\|\cdot\|}\{\alpha(x)(1\otimes \lambda(a)): x\in X, a \in \mathcal{O}(\G)\}\subseteq X \ovot B(L^2(\G))$.
\end{defn}

\begin{defn}\label{Fubini}
Let $(X, \alpha)$ be a $\wG$-operator system. We define the Fubini crossed product to be the operator system
$$ X \rtimes_{\alpha, \mathcal{F}} \wG:= \{z \in X \ovot B(L^2(\G)): (\id_X \ovot \hat{\Delta}_l)(z)= (\alpha \ovot \id)(z)\}.$$
\end{defn}

\begin{rem}
    \normalfont{If $X$ is a von Neumann algebra and $\alpha: X \to X \ovot \mathscr{L}^\infty(\wG)$ is a normal unital $*$-homomorphism, then $X\rtimes_{\alpha, \mathcal{F}} \wG$ agrees with the usual von Neumann crossed product. However, there is another valuable point of view. We could view $X\rtimes_{\alpha, \mathcal{F}} \wG$ as the cotensor product (in the setting of operator systems) of $X$ and the von Neumann algebra $B(L^2(\G))$ with respect to the left action $\hat{\Delta}_l$ and the right action $\alpha$. See also \cite[Section 7]{cotensor} for the use of the cotensor product in the context of compact quantum groups acting on $C^*$-algebras and von Neumann algebras. We use this point of view to construct dual actions on the Fubini crossed product, see Theorem \ref{important action}, Proposition \ref{action III} and also \cite[Lemma 7.2]{cotensor}.
    $\Di$}
\end{rem}

\begin{ex}
    \normalfont{Let $\Gamma$ be a discrete group and $\G = \widehat{\Gamma}$ be its dual CQG as in Example \ref{example}. If $\alpha: X \to X \ovot \ell^\infty(\Gamma)$ is a coaction as in Definition \ref{action DQG} then we obtain a group action $\beta: \Gamma \to\operatorname{Aut}(X)$ by
    $$\beta_g:= (\id \ovot \ev_g)\circ \alpha, \quad g \in \Gamma$$
    where $\ev_g(f) = f(g^{-1})$ for $g\in \Gamma$ and $f\in \ell^\infty(\Gamma)$. If $z \in X \ovot B(\ell^2(\Gamma))$ we can think of it as an infinite matrix $[z_{g,h}]_{g,h \in \Gamma}$ with entries in $X$, where $z_{g,h}= (\id \ovot \omega_{\delta_g, \delta_h})(z) \in X$. It then follows that 
    \begin{align*}
        z \in X \rtimes_{\alpha, \mathcal{F}} \wG &\iff (\id_X \ovot \hat{\Delta}_l)(z) = (\alpha \ovot \id_{B(\ell^2(\Gamma))})(z)\\
        &\iff \forall g\in \Gamma: (\id \ovot \ev_g \ovot \id)(\id \ovot \hat{\Delta}_l)(z) = (\id \ovot \ev_g \ovot \id)(\alpha \ovot \id)(z)\\
        &\iff \forall g \in \Gamma: (1 \otimes (\ev_g\ovot \id)(W_{\wG})^*)z(1 \otimes (\ev_g\ovot \id)(W_{\wG})) = (\beta_{g}\ovot \id)(z)\\
        &\iff \forall g \in \Gamma: (1\otimes \rho_g^*)z(1\otimes \rho_g) = (\beta_{g}\ovot \id)(z)\\
        & \iff \forall g,h,k \in \Gamma: (\id \ovot \omega_{\delta_h, \delta_k})(\id \ovot \Ad(\rho_g^*))(z) = (\id \ovot \omega_{\delta_h, \delta_k})(\beta_{g}\ovot \id)(z)\\
        &\iff \forall g,h,k \in \Gamma: (\id \ovot \omega_{\delta_{hg^{-1}}, \delta_{kg^{-1}}})(z) = \beta_{g}((\id \ovot \omega_{\delta_h, \delta_k})(z))\\
        &\iff \forall g,h,k \in \Gamma: z_{hg^{-1}, kg^{-1}}= \beta_{g}(z_{h,k}) \ .
    \end{align*}
    Consequently,
    $$X\rtimes_{\alpha, \mathcal{F}}\wG= \{z\in X \ovot B(\ell^2(\Gamma)): \forall g,h,k\in \Gamma: \beta_{g^{-1}}(z_{h,k}) = z_{hg, kg}\} \ ,$$
    and we recover Hamana's Fubini crossed product  \cite[II, Section 3]{Ham3}.
    $\Di$}
\end{ex}

\begin{ex}\normalfont{Let $X\subseteq B(\H)$ be an operator system and $\tau: X \to X \ovot \mathscr{L}^\infty(\wG)$ the trivial action $\tau(x)= x \otimes 1$. We claim that $X \rtimes_{\tau, \mathcal{F}} \wG = X \ovot \mathscr{L}^\infty(\G).$
Indeed, it is easily seen that the right hand side is contained in the left hand side. Conversely, using that 
$\{(\omega \ovot \id)(W_{\wG}): \omega \in B(L^2(\G))_*\}$ generates the von Neumann algebra $\mathscr{L}^\infty(\G)'$, it readily follows that
$$\{x\in B(L^2(\G)): W_{\wG}^*(1\otimes x)W_{\wG}= 1 \otimes x\}= \mathscr{L}^\infty(\G).$$
Thus, if $z\in X\rtimes_{\tau, \mathcal{F}} \wG$, then $W_{\wG, 23}^*z_{13}W_{\wG, 23} = z_{13}$. If $\omega \in B(\H)_*$, then applying the slice map $\omega \ovot \id\ovot \id$ to this equality yields 
$$W_{\wG}^*(1\otimes (\omega \ovot \id)(z))W_{\wG}= 1\otimes (\omega \ovot \id)(z) \ ,$$
so that $(\omega \ovot \id)(z)\in \mathscr{L}^\infty(\G)$. By definition of the Fubini tensor product, this means precisely that $z \in X \ovot \mathscr{L}^\infty(\G)$. Note also that trivially 
$X \rtimes_{\tau,r} \wG = X \otimes C_r(\G)$. $\Di$}
\end{ex}
\begin{prop} Let $z\in X \rtimes_{\alpha, \mathcal{F}} \wG$. Then $(1\otimes \lambda(a))z(1\otimes \lambda(b))\in  X \rtimes_{\alpha, \mathcal{F}} \wG$ for all $a,b\in \mathcal{O}(\G)$. In particular, $X\rtimes_{\alpha, \mathcal{F}}\wG$ has a $C_r(\G)$-bimodule structure. Moreover, $X\rtimes_{\alpha,r}\wG \subseteq  X \rtimes_{\alpha, \mathcal{F}} \wG$.
\end{prop}

\begin{proof} We calculate for $a,b \in \mathcal{O}(\G)$ and $z\in  X \rtimes_{\alpha, \mathcal{F}} \wG$ that 
\begin{align*}
        (\id \ovot \hat{\Delta}_l)((1\otimes \lambda(a))z(1\otimes \lambda(b)))&= (\id \ovot \hat{\Delta}_l)(1\otimes\lambda(a)) (\id \ovot \hat{\Delta}_l)(z)(\id \ovot \hat{\Delta}_l)(1\otimes \lambda(b))\\
        &=(1\otimes 1\otimes \lambda(a))(\alpha \ovot \id)(z)(1\otimes 1\otimes \lambda(b))\\
        &= (\alpha \ovot \id)((1\otimes \lambda(a))z(1\otimes \lambda(b))) \ ,
    \end{align*}
    so that $(1\otimes \lambda(a))z(1\otimes \lambda(b))\in  X \rtimes_{\alpha, \mathcal{F}} \wG$. The inclusion $\alpha(X)\subseteq X\rtimes_{\alpha, \mathcal{F}} \wG$ is immediately clear from the coaction property. 
\end{proof}

\begin{lma}
     $(X \rtimes_{\alpha, \mathcal{F}} \wG) \cap (X \ovot \mathscr{L}^\infty(\wG)) = \alpha(X)$.
\end{lma}
\begin{proof}
    Assume that $z\in  (X \rtimes_{\alpha, \mathcal{F}} \wG)\cap (X \ovot \mathscr{L}^\infty(\wG)).$ Then $(\alpha \ovot \id)(z) = (\id \ovot \hat{\Delta}_l)(z)= (\id \ovot \hat{\Delta})(z)$. Applying the slice map $\id \ovot \id \ovot \hat{\epsilon}$ to this equality, we get $z= \alpha((\id \ovot \hat{\epsilon})(z)).$
\end{proof}

\begin{thm}\label{important action} If $(X, \alpha)$ is a $\wG$-operator system, then the $\G$-$W^*$-action $\id_X\ovot \Delta: X \ovot B(L^2(\G))\to X \ovot B(L^2(\G))\ovot \mathscr{L}^\infty(\G)$ restricts to a $\G$-$W^*$-action $\id \ovot \Delta: X\rtimes_{\alpha,\mathcal{F}} \wG\to (X\rtimes_{\alpha, \mathcal{F}} \wG)\ovot \mathscr{L}^\infty(\G)$ and a $\G$-$C^*$-action $\id \ovot \Delta: X \rtimes_{\alpha,r}\wG \to (X\rtimes_{\alpha,r}\wG)\otimes C_r(\G)$. Moreover, 
$$\mathcal{R}_{\operatorname{alg}}(X\rtimes_{\alpha,\mathcal{F}} \wG, \id \ovot \Delta) = \lspan\{\alpha(x)(1\otimes \lambda(a)): x \in X, a \in \mathcal{O}(\G)\}, \quad \operatorname{Fix}(X\rtimes_{\alpha,\mathcal{F}} \wG, \id \ovot \Delta)= \alpha(X).$$ 
\end{thm}
\begin{proof} We first note that the diagram
$$
\begin{tikzcd}
                                                                        & B(L^2(\G)) \arrow[ld, "\hat{\Delta}_l"'] \arrow[rd, "\Delta"]       &                                                                              \\
\mathscr{L}^\infty(\wG)\ovot B(L^2(\G)) \arrow[rd, "\id \ovot \Delta"'] &                                                                     & B(L^2(\G))\ovot \mathscr{L}^\infty(\G) \arrow[ld, "\hat{\Delta}_l\ovot \id"] \\
                                                                        & \mathscr{L}^\infty(\wG)\ovot B(L^2(\G))\ovot \mathscr{L}^\infty(\G) &                                                                             
\end{tikzcd}$$
commutes. To see this, note that commutativity of this diagram is equivalent to 
    $$V_{\G, 23}W_{\wG, 12}^*(1\otimes x \otimes 1)W_{\wG, 12}V_{\G, 23}^* = W_{\wG, 12}^*V_{\G, 23}(1\otimes x \otimes 1)V_{\G, 23}^*W_{\wG, 12}, \quad x \in B(L^2(\G)).$$
By weak density, it suffices to show this equality for $x\in B_0(L^2(\G))$. We have $$[\hat{\lambda}(\omega)\lambda(a): \omega \in c_c(\wG), a \in \mathcal{O}(\G)]= B_0(L^2(\G)).$$
If $a\in \mathcal{O}(\G)$ and $\omega \in c_c(\wG)$, it therefore is sufficient to prove the desired equality for operators $x=\hat{\lambda}(\omega)\lambda(a)$. Since $V_{\G} \in \mathscr{L}^\infty(\wG)'\ovot \mathscr{L}^\infty(\G)$ and $W_{\wG}\in \mathscr{L}^\infty(\wG)\ovot \mathscr{L}^\infty(\G)'$, a calculation shows
\begin{align*}
    V_{\G, 23}W_{\wG, 12}^*(1\otimes \hat{\lambda}(\omega)\lambda(a)\otimes 1)W_{\wG, 12}V_{\G, 23}^* &= (\hat{\Delta}(\hat{\lambda}(\omega))\otimes 1)(1\otimes \Delta(\lambda(a)))\\
    &= W_{\wG, 12}^*V_{\G, 23}(1\otimes \hat{\lambda}(\omega)\lambda(a) \otimes 1)V_{\G, 23}^*W_{\wG, 12} \ ,
\end{align*}
so we conclude that the above diagram is commutative. Let $\omega \in B(\H)_*$ where $X\subseteq B(\H)$ and $z \in X\rtimes_{\alpha,\mathcal{F}}\wG$. We calculate
    \begin{align*}
        (\alpha \ovot \id)(\id \ovot \id \ovot \omega)(\id \ovot \Delta)(z)&= (\id \ovot \id \ovot \id \ovot \omega)(\alpha \ovot \id \ovot \id)(\id \ovot \Delta)(z)\\
        &= (\id \ovot \id \ovot \id \ovot \omega)(\id \ovot \id \ovot \Delta)(\alpha \ovot \id)(z)\\
        &= (\id \ovot \id \ovot \id \ovot \omega)(\id \ovot \id \ovot \Delta)(\id \ovot \hat{\Delta}_l)(z)\\
        &= (\id \ovot \id \ovot \id \ovot \omega)(\id \ovot \hat{\Delta}_l \ovot \id)(\id \ovot \Delta)(z)\\
        &= (\id \ovot \hat{\Delta}_l)(\id \ovot \id \ovot \omega)(\id \ovot \Delta)(z) \ ,
    \end{align*}
    so it follows that $(\id \ovot \id \ovot \omega)(\id \ovot \Delta)(z)\in X \rtimes_{\alpha,\mathcal{F}}\wG$. By definition of the Fubini tensor product, this means that $$(\id \ovot \Delta)(z) \in [(X\rtimes_{\alpha,\mathcal{F}}\wG) \ovot B(L^2(\G))]\cap[X\ovot B(L^2(\G)) \ovot \mathscr{L}^\infty(\G)]= (X\rtimes_{\alpha,\mathcal{F}}\wG)\ovot \mathscr{L}^\infty(\G).$$ Thus, $\id \ovot \Delta$ restricts to a $\G$-$W^*$-action on the Fubini crossed product. 

Assume that $z\in \mathcal{R}_{\operatorname{alg}}(X\rtimes_{\alpha, \mathcal{F}} \wG)$. Then by Proposition \ref{important action II} also $z\in \mathcal{R}_{\operatorname{alg}}(X \ovot B(L^2(\G))) = (X \ovot \mathscr{L}^\infty(\wG))(1\otimes \mathcal{O}(\G))$. We can therefore write $$z= \sum_{\pi \in \Irr(\G)} \sum_{i,j=1}^{n_\pi} z_{ij}^\pi(1\otimes u_{ij}^\pi) \ ,$$
    where $z_{ij}^\pi \in X \ovot \mathscr{L}^\infty(\wG)$. By assumption, 
    $$(X\rtimes_{\alpha,\mathcal{F}} \wG) \odot \mathcal{O}(\G)\ni (\id \ovot \Delta)(z) = \sum_{\pi \in \Irr(\G)} \sum_{i,j,k=1}^{n_\pi} z_{ij}^\pi(1\otimes u_{ik}^\pi)\otimes u_{kj}^\pi.$$
    It follows that $\sum_{i=1}^{n_{\pi'}} z_{ij'}^{\pi'}(1\otimes u_{ik'}^{\pi'})\in X\rtimes_{\alpha,\mathcal{F}} \wG$ for fixed $\pi', j',k'$. Since $X\rtimes_{\alpha,\mathcal{F}} \wG$ is a $C_r(\G)$-bimodule, we see that
    $$X\rtimes_{\alpha,\mathcal{F}} \wG \ni \sum_{k'=1}^{n_{\pi'}}\left(\sum_{i=1}^{n_{\pi'}} z_{ij'}^{\pi'} (1\otimes u_{ik'}^{\pi'})\right)(1\otimes (u_{i'k'}^{\pi'})^*) = z_{i'j'}^{\pi'}.$$
    It follows that $z_{ij}^\pi \in (X\rtimes_{\alpha,\mathcal{F}} \wG)\cap (X\ovot \mathscr{L}^\infty(\wG))= \alpha(X)$ for all $\pi, i,j$. Thus, we can write $z_{ij}^\pi = \alpha(x_{ij}^\pi)$ for unique $x_{ij}^\pi\in X$ and it follows that
    $$z= \sum_{\pi \in \Irr(\G)} \sum_{i,j=1}^{n_\pi}\alpha(x_{ij}^\pi)(1\otimes u_{ij}^\pi).$$
    The theorem follows.    
\end{proof}

We then find the following pleasant relation between the Fubini crossed product and the reduced crossed product:
\begin{cor} $\mathcal{R}(X\rtimes_{\alpha,\mathcal{F}} \wG, \id \ovot \Delta)= X\rtimes_{\alpha,r} \wG$.
\end{cor}

As expected, $\wG$-equivariant ucp maps induce $\G$-equivariant maps on the associated crossed products. The proof of the following lemma is a triviality.
\begin{lma}
    Let $(X, \alpha)$ and $(Y, \beta)$ be $\wG$-operator systems and $\phi: X \to Y$ a $\wG$-equivariant ucp map. Then the map $\phi \ovot \id: X \ovot B(L^2(\G))\to Y \ovot B(L^2(\G))$ restricts to a $\G$-$W^*$-ucp map $\phi \rtimes_{\mathcal{F}}\wG: X\rtimes_{\alpha, \mathcal{F}}\wG\to X\rtimes_{\beta, \mathcal{F}}\wG$ and a $\G$-$C^*$-ucp map $\phi\rtimes_r \wG: X\rtimes_{\alpha, r}\wG\to X\rtimes_{\beta, r}\wG$ and we have
    $(\phi \otimes \id)(\alpha(x)(1\otimes \lambda(a)) = \beta(\phi(x))(1\otimes \lambda(a))$
    for all $x\in X$ and all $a\in \mathcal{O}(\G)$. 
\end{lma}

Besides a canonical $\G$-$W^*$-action, the Fubini crossed product also carries a canonical $\wG$-action:

\begin{prop}\label{action III}
    If $(X, \alpha)$ is a $\wG$-operator system, then the action $\id_X \ovot \hat{\Delta}_r: X \ovot B(L^2(\G))\to X \ovot B(L^2(\G))\ovot \mathscr{L}^\infty(\wG)$ restricts to an action 
    $X \rtimes_{\alpha, \mathcal{F}}\wG \to (X \rtimes_{\alpha, \mathcal{F}}\wG) \ovot \mathscr{L}^\infty(\wG)$.
\end{prop}
\begin{proof}
    It suffices to note that the diagram 
    $$
\begin{tikzcd}
                                                                                & B(L^2(\G)) \arrow[rd, "\hat{\Delta}_r"] \arrow[ld, "\hat{\Delta}_l"'] &                                                                               \\
\mathscr{L}^\infty(\wG)\ovot B(L^2(\G)) \arrow[rd, "\id \ovot \hat{\Delta}_r"'] &                                                                       & B(L^2(\G))\ovot \mathscr{L}^\infty(\wG) \arrow[ld, "\hat{\Delta}_l\ovot \id"] \\
                                                                                & \mathscr{L}^\infty(\wG)\ovot B(L^2(\G))\ovot \mathscr{L}^\infty(\wG)  &                                                                              
\end{tikzcd}$$
commutes. As in the proof of Theorem \ref{important action}, this implies that $\id \ovot \hat{\Delta}_r$ restricts to a right $\wG$-action on the Fubini crossed product.
\end{proof}

The Fubini crossed product admits a canonical conditional expectation.

\begin{prop}\label{conditional expectation} The map $E: X \rtimes_{\alpha,\mathcal{F}}\wG\to \alpha(X)$ defined by $E:= \alpha \circ (\id \ovot \omega_{\xi_\G})$ is a faithful ucp conditional expectation on $\alpha(X)$. The following properties are satisfied:
\begin{enumerate}
    \item $E = (\id \ovot \id \ovot \omega_{\xi_\G})\circ (\id \ovot \Delta)$
    \item $E: (X\rtimes_{\alpha, \mathcal{F}}\wG, \id \ovot \Delta) \to (\alpha(X), \tau)$ is $\G$-$W^*$-equivariant, where $\tau: \alpha(X)\to \alpha(X)\ovot \mathscr{L}^\infty(\G)$ is the trivial action on $\alpha(X)$.
    \item If $\G$ is of Kac-type, then $E: (X\rtimes_{\alpha, \mathcal{F}}\wG, \id \ovot \hat{\Delta}_r)\to (\alpha(X), \id \ovot \hat{\Delta})$ is $\wG$-equivariant.
\end{enumerate}
\end{prop}
\begin{proof} We calculate for $x\in X$:
\begin{align*}
    E(\alpha(x)) = \alpha((\id \ovot \omega_{\xi_\G})\alpha(x))= (\id \ovot \id \ovot \omega_{\xi_\G})(\alpha \ovot \id)\alpha(x)= (\id \ovot \id \ovot \omega_{\xi_\G})(\id \ovot \hat{\Delta})(\alpha(x))= \alpha(x)
\end{align*}
so $E$ is a ucp conditional expectation onto $\alpha(X)$.

If $z\in X\rtimes_{\alpha, \mathcal{F}} \wG$, then $(\id \ovot \Delta)(z)\in (X\rtimes_{\alpha, \mathcal{F}}\wG)\ovot \mathscr{L}^\infty(\G)$ so that
    \begin{align*}
        (\id \ovot \Delta) (\id \ovot \id \ovot \omega_{\xi_\G})(\id \ovot \Delta)(z)&= (\id \ovot \id \ovot \id \ovot \omega_{\xi_\G})(\id \ovot \Delta \ovot \id)(\id \ovot \Delta)(z)\\
        &= (\id \ovot \id \ovot \id \ovot \omega_{\xi_\G})(\id \ovot \id \ovot \Delta)(\id \ovot \Delta)(z)\\
        &= (\id \ovot \id \ovot \omega_{\xi_\G})(\id \ovot \Delta)(z)\otimes 1.
    \end{align*}
    It follows that $(\id \ovot \id \ovot \omega_{\xi_\G})(\id \ovot \Delta)(z) \in \operatorname{Fix}(X\rtimes_{\alpha, \mathcal{F}} \wG, \id \ovot \Delta) = \alpha(X)$. Therefore, there is $x\in X$ with $(\id \ovot \id \ovot \omega_{\xi_\G})(\id \ovot \Delta)(z) = \alpha(x)$. Applying $\id \ovot \omega_{\xi_\G}$ to this expression, we get $(\id \ovot \omega_{\xi_\G})(z) = x$. Therefore, 
\begin{align*}
    E(z)= \alpha((\id \ovot \omega_{\xi_\G})(z)) = (\id \ovot \id \ovot \omega_{\xi_\G})(\id \ovot \Delta)(z).
\end{align*}
Since $\omega_{\xi_\G}$ is faithful on $\mathscr{L}^\infty(\G)$, it follows that $E$ is faithful as well. Note that if $a,b\in \mathcal{O}(\G)$ and $x\in B(L^2(\G))$, we have
\begin{align*}
    \omega_{\Lambda(a), \Lambda(b)}((\omega_{\xi_\G}\ovot \id)\Delta(x)) = \langle V_\G^*(\Lambda(1)\otimes \Lambda(b)), (x\otimes 1)V_\G^*(\Lambda(1)\otimes \Lambda(b))\rangle = \omega_{\Lambda(a), \Lambda(b)}(\omega_{\xi_\G}(x)1) \ ,
\end{align*}
so it follows that $(\omega_{\xi_\G}\ovot \id)\Delta(x)= \omega_{\xi_\G}(x)1.$ Therefore, 
\begin{align*}
    (E\ovot \id)(\id \ovot \Delta)(z)&= (\alpha \ovot \id)(\id \ovot \omega_{\xi_\G}\ovot \id)(\id \ovot \Delta)(z)= (\alpha \ovot \id)((\id \ovot \omega_{\xi_\G})(z)\otimes 1) = E(z)\otimes 1.
\end{align*}
Finally, if $\G$ is of Kac-type, then to conclude $\wG$-equivariance of $E$, it suffices to show that the map
$$(X \rtimes_{\alpha, \mathcal{F}}\wG, \id \ovot \hat{\Delta}_r) \to (X, \alpha): z \mapsto (\id \ovot \omega_{\xi_\G})(z)$$ is $\wG$-equivariant. Using the identity \eqref{rem:Kacidentity}, we see that
\begin{align*}
    (\id \ovot \omega_{\xi_\G}\ovot \id)(\id \ovot \hat{\Delta}_r)(z) &= (\id \ovot \id \ovot \omega_{\xi_\G})(\id \ovot \hat{\Delta}_l)(z)\\
    &= (\id \ovot \id \ovot \omega_{\xi_\G})(\alpha \ovot \id)(z)= \alpha((\id \ovot \omega_{\xi_\G})(z)) \ ,
\end{align*}
and the proposition is proven.
\end{proof}

\begin{cor}\label{yeeha}
    If $X\rtimes_{\alpha,r} \wG$ is $\G$-$C^*$-injective, then $X$ is injective as an operator system. If $X\rtimes_{\alpha, \mathcal{F}} \wG$ is $\G$-$W^*$-injective, then $X$ is injective as an operator system.
\end{cor}
\begin{proof}
    By Lemma \ref{conditional expectation}, there is a $\G$-$W^*$-equivariant conditional expectation $E: (X\rtimes_{\alpha, \mathcal{F}} \wG, \id \ovot \Delta)\to (\alpha(X), \tau)\cong (X, \tau)$. It follows that $(X, \tau)$ is $\G$-$W^*$-injective, i.e. $X$ is injective as an operator system. The other claim is proven in the same way.
\end{proof}

\section{Equivariant injectivity of crossed products}

In this section, we start from a $\wG$-operator system $(X, \alpha)$ and show that $\wG$-injectivity of this operator system is equivalent with $\G$-injectivity of the associated crossed products. As an application, we give a characterisation of amenability of the DQG $\wG$ in terms of equivariant injectivity of associated operator algebras.

The following crucial lemma ensures that certain maps are automatically $\wG$-equivariant:

\begin{lma} \label{multiplicative domain argument} Let $X\subseteq B(\H)$ be an operator system and $\psi: X \ovot B(L^2(\G))\to X \ovot B(L^2(\G))$ a ucp map such that $\psi(1\otimes \lambda(a)) = 1\otimes \lambda(a)$ for all $a\in \mathcal{O}(\G)$. Then $\psi$ automatically preserves the right coaction $\id_X \ovot \hat{\Delta}_r: X \ovot B(L^2(\G)) \to X \ovot B(L^2(\G))\ovot \mathscr{L}^\infty(\wG)$.
\end{lma}
\begin{proof} By Arveson's extension theorem, $\psi$ extends to a ucp map $$\widetilde{\psi}: B(\H \otimes L^2(\G)) \to B(\H \otimes L^2(\G)).$$ 
    By the argument in \cite[Lemma 6.2]{KKSV}, $\widetilde{\psi}$ preserves the coaction $\id_{B(\H)}\ovot \hat{\Delta}_r$ from which the lemma immediately follows.
\end{proof}

We now come to the main results of this paper. Similar results were obtained in the unpublished preprint \cite[Theorem 7.5 \& Corollary 7.7]{MM} in the context of von Neumann algebras, using a non-trivial notion of braided tensor product of von Neumann algebras. The Kac-type case in the following theorem generalises the original result \cite[II, Lemma 3.1]{Ham3}.

\begin{thm}\label{important}
    Let $(X, \alpha)$ be a $\wG$-operator system. The existence of the following maps are equivalent:
    \begin{enumerate}
        \item A $\wG$-equivariant ucp conditional expectation $(X \ovot \mathscr{L}^\infty(\wG), \id \ovot \hat{\Delta})\to (\alpha(X), \id \ovot \hat{\Delta})$.
        \item A $\wG$-equivariant ucp conditional expectation $(X\ovot B(L^2(\G)), \id \ovot \hat{\Delta}_r) \to (\alpha(X), \id \ovot \hat{\Delta})$.
        \item A $\G$-$W^*$-equivariant ucp conditional expectation $(X \ovot B(L^2(\G)), \id \ovot \Delta)\to (X\rtimes_{\alpha, \mathcal{F}} \wG, \id \ovot \Delta).$
        \item  A $\G$-$C^*$-equivariant ucp conditional expectation $\mathcal{R}(X \ovot B(L^2(\G)), \id \ovot \Delta)\to (X\rtimes_{\alpha,r}\wG, \id \ovot \Delta).$
    \end{enumerate}
    If $\G$ is of Kac-type, the existence of these maps is equivalent with the existence of:
    \begin{enumerate}  \setcounter{enumi}{4}
        \item A ucp conditional expectation $X\ovot B(L^2(\G))\to X\rtimes_{\alpha, \mathcal{F}} \wG.$
    \end{enumerate}
\end{thm}
\begin{proof} We will assume that $X\subseteq B(\H)$. 

$(1)\iff (2)$ It suffices to note that the map
$$B(L^2(\G))\to \mathscr{L}^\infty(\wG): x\mapsto (\omega_{\xi_\G}\ovot \id)(\hat{\Delta}_r(x))$$
is a ucp $\wG$-equivariant conditional expectation.

    $(1)\implies (3)+(4)$ By assumption, there exists an equivariant ucp map $\phi: (X \ovot \mathscr{L}^\infty(\wG), \id \ovot \hat{\Delta})\to (X, \alpha)$ such that $\phi\circ \alpha= \id_X$. Define
    $$\psi:= (\phi \ovot \id)\circ (\id \ovot \hat{\Delta}_l): X \ovot B(L^2(\G)) \to X \ovot B(L^2(\G)).$$

    If $z \in X\rtimes_{\alpha,\mathcal{F}} \wG$, then
    $$\psi(z) = (\phi \ovot \id)(\id \ovot \hat{\Delta}_l)(z)=(\phi\ovot \id)(\alpha \ovot \id)(z) = z.$$
    Moreover, if $z\in X \ovot B(L^2(\G))$, then
    \begin{align*}
        (\alpha \ovot \id)(\psi(z))&= (\alpha \ovot \id)(\phi \ovot \id)(\id \ovot \hat{\Delta}_l)(z)\\
        &= (\phi \ovot \id \ovot \id)(\id \ovot \hat{\Delta}\ovot \id)(\id \ovot \hat{\Delta}_l)(z)\\
        &= (\phi \ovot \id \ovot \id)(\id \ovot \id \ovot \hat{\Delta}_l)(\id \ovot \hat{\Delta}_l)(z)\\
        &= (\id \ovot \hat{\Delta}_l)(\phi \ovot \id)(\id \ovot \hat{\Delta})(z)\\
        &= (\id \ovot \hat{\Delta}_l)(\psi(z)) \ ,
    \end{align*}
    so that $\psi(z)\in X\rtimes_{\alpha,\mathcal{F}} \wG$. It follows that $\psi: X \ovot B(L^2(\G))\to X\rtimes_{\alpha, \mathcal{F}} \wG$ is a 
 ucp conditional expectation.
 We claim that $\psi$ is $\G$-$W^*$-equivariant. Indeed, since $(\id \ovot \Delta)\hat{\Delta}_l= (\hat{\Delta}_l\ovot \id)\Delta$, as shown in the proof of Theorem \ref{important action}, we can calculate for $z\in X \ovot B(L^2(\G))$ that
    \begin{align*}
        (\id \ovot \Delta)(\psi(z)) &= (\id \ovot \Delta)(\phi \ovot \id)(\id \ovot \hat{\Delta}_l)(z)\\
        &=(\phi \ovot \id \ovot \id)(\id \ovot \id \ovot \Delta)(\id \ovot \hat{\Delta}_l)(z)\\
        &= (\phi \ovot \id \ovot \id)(\id \ovot \hat{\Delta}_l\ovot \id)(\id \ovot \Delta)(z)\\
        &= (\psi \ovot \id)(\id \ovot \Delta)(z).
    \end{align*}

     Moreover, since $\psi$ is $\G$-$W^*$-equivariant, it restricts to a $\G$-$C^*$-equivariant ucp conditional expectation
    $$\mathcal{R}(X \ovot B(L^2(\G)))\to \mathcal{R}(X\rtimes_{\alpha,\mathcal{F}} \wG) = X\rtimes_{\alpha,r} \wG.$$

    $(3)\implies (1)$ Let a $\G$-$W^*$-equivariant conditional expectation $P: X \ovot B(L^2(\G))\to X \rtimes_{\alpha, \mathcal{F}}\wG$ be given. By Lemma \ref{multiplicative domain argument}, $P$ automatically preserves the coaction $\id_X \ovot \hat{\Delta}_r$. Since $P$ preserves the $\id \ovot \Delta$-coaction, it restricts to a ucp conditional expectation 
    $$Q: \operatorname{Fix}(X \ovot B(L^2(\G)), \id \ovot \Delta) = X \ovot \mathscr{L}^\infty(\wG) \to \operatorname{Fix}(X\rtimes_{\alpha,\mathcal{F}} \wG, \id \ovot \Delta)= \alpha(X)$$
    which clearly preserves the coaction $\id \ovot \hat{\Delta}$.

    $(4)\implies (1)$ Let a $\G$-$C^*$-equivariant ucp conditional expectation
    $$P: (\mathcal{R}(X \ovot B(L^2(\G))), \id \ovot \Delta)\to (X \rtimes_{\alpha,r} \wG, \id \ovot \Delta)$$ be given. By Arveson's extension theorem and Lemma \ref{multiplicative domain argument}, we see that $P$ has a ucp $\wG$-equivariant extension 
    $$\widetilde{P}: (B(\H) \ovot B(L^2(\G)), \id \ovot \hat{\Delta}_r)\to (B(\H) \ovot B(L^2(\G)), \id \ovot \hat{\Delta}_r).$$
    Since $P$ preserves the coaction $\id \ovot \Delta$, it restricts to a ucp conditional expectation 
    $$Q: X \ovot \mathscr{L}^\infty(\wG) \to \alpha(X) \ ,$$
    which is $\wG$-equivariant since $\widetilde{P}$ is $\wG$-equivariant.

    $(5)\implies (1)$ (Under the additional assumption that $\G$ is of Kac type.) By assumption, there exists a ucp conditional expectation 
    $$P: X \ovot B(L^2(\G))\to X\rtimes_{\alpha, \mathcal{F}}\wG \ ,$$ which automatically preserves the coaction $\id \ovot \hat{\Delta}_r$. Since $\G$ is of Kac type, Proposition \ref{conditional expectation} tells us that $E: (X\rtimes_{\alpha, \mathcal{F}}\wG, \id \ovot \hat{\Delta}_r) \to (\alpha(X), \id \ovot \hat{\Delta})$ is $\wG$-equivariant. The composition 
    $$
\begin{tikzcd}
X\ovot \mathscr{L}^\infty(\wG) \arrow[rr, "\subseteq "] &  & X\ovot B(L^2(\G)) \arrow[rr, "P"] &  & {X\rtimes_{\alpha, \mathcal{F}}\wG} \arrow[rr, "E"] &  & \alpha(X)
\end{tikzcd}$$
is therefore $\wG$-equivariant. The result follows.
\end{proof}

We now obtain the following result:
\begin{thm}\label{main result}
    Let $(X, \alpha)$ be a $\wG$-operator system. The following statements are equivalent:
    \begin{enumerate}
        \item $(X, \alpha)$ is $\wG$-injective.
        \item $(X\rtimes_{\alpha,\mathcal{F}} \wG, \id \ovot \Delta)$ is $\G$-$W^*$-injective.
        \item $(X \rtimes_{\alpha,r} \wG, \id \ovot \Delta)$ is $\G$-$C^*$-injective.
            \end{enumerate}
                    Moreover, if $\G$ is of Kac-type, then $(X, \alpha)$ is $\wG$-injective if and only if $X\rtimes_{\alpha,\mathcal{F}} \wG$ is injective as an operator system.

\end{thm}

\begin{proof}
    $(1)\implies (2)$ Assume that $(X, \alpha)$ is $\wG$-injective. By Proposition \ref{hat injectivity} and Theorem \ref{important}, $X$ is injective and there exists a $\G$-$W^*$-ucp conditional expectation $$(X\ovot B(L^2(\G)), \id \ovot \Delta)\to (X\rtimes_{\alpha, \mathcal{F}}\wG, \id \ovot \Delta).$$
By Lemma \ref{G-injective systems}, it follows that $(X\rtimes_{\alpha, \mathcal{F}}\wG, \id \ovot \Delta)$ is $\G$-$W^*$-injective.

    $(2)\implies (3)$ Immediate by Lemma \ref{regular injectivity}.

    $(3)\implies (1)$ Note that $X\rtimes_{\alpha, r}\wG\subseteq \mathcal{R}(X \ovot B(L^2(\G)), \id_X \ovot \Delta)$, so by $\G$-$C^*$-injectivity, there is a ucp $\G$-$C^*$-equivariant conditional expectation 
    $$(\mathcal{R}(X \ovot B(L^2(\G))), \id \ovot \Delta)\to (X\rtimes_{\alpha,r}\wG, \id \ovot \Delta).$$
    Combining Theorem \ref{important}, Corollary \ref{yeeha} and Proposition \ref{hat injectivity}, we see that $(X, \alpha)$ is $\wG$-injective.

    The assertion in the Kac case is proven similarly, using Theorem \ref{important}.
\end{proof}

We now apply Theorem \ref{main result} to the $\wG$-operator system $\mathbb{C}$ with the trival action. In the Kac case, we then recover a result established in \cite{Ruan}. 
\begin{cor}\label{amenability}
    The following statements are equivalent:
    \begin{enumerate}
        \item $\wG$ is amenable.
        \item There exists a $\wG$-equivariant state $(B(L^2(\G)), \hat{\Delta}_r)\to (\mathbb{C}, \tau).$
        \item $(\mathscr{L}^\infty(\G), \Delta)$ is $\G$-$W^*$-injective.
        \item $(C_r(\G), \Delta)$ is $\G$-$C^*$-injective.
            \end{enumerate}
                    Moreover, if $\G$ is of Kac type, then $\wG$ is amenable if and only if $\mathscr{L}^\infty(\G)$ is injective.
\end{cor}

In \cite{Crann}, the notion of $1$-injectivity of $\mathscr{L}^\infty(\G)$ as an operator $L^1(\G)$-module is introduced (in the context of general locally compact quantum groups). The main result of this paper asserts that this notion of $1$-injectivity is equivalent with amenability of $\wG$, which resembles one of the equivalences in Corollary \ref{amenability}.

\begin{rem}\label{bad}\normalfont{Let $\G$ be any coamenable CQG with separable infinite-dimensional function algebra $C(\G)$ (for instance, consider $\G = \widehat{\Gamma}$ where $\Gamma$ is an amenable, countably infinite group). Then $C_r(\G)$ is $\G$-$C^*$-injective but not injective as an operator system.$\Di$}
\end{rem}

\section{Extensions and injective envelopes of crossed products} 

In this section, we use a proposition from \cite{HHN} to show that every $\G\text{-}C^*$-operator system $X$ admits a $\G\text{-}C^*$-injective envelope $I_{\G}^{C^*}(X)$ and that every $\wG$-operator system $X$ admits a $\wG$-injective envelope $I_{\wG}(X)$. See also \cite[Section 3]{HY} where this lemma is applied to show that Yetter-Drinfeld algebras admit equivariant injective envelopes and see also \cite[Section 4]{KKSV} where similar techniques are used to construct the non-commutative Furstenberg boundary associated to a DQG.

We then show that if $(X, \alpha)$ is a $\wG$-operator system, then 
$I_{\wG}(X) \rtimes_{r} \wG$ is the $\G$-$C^*$-injective envelope of the $\G$-$C^*$-operator system $(X\rtimes_{\alpha, r} \wG, \id \otimes \Delta)$.

We start by giving the following straightforward definition:

\begin{defn} Let $X$ be a $\G$-$C^*$-operator system. A pair $(Y, \iota)$ is called $\G$-$C^*$-extension of $X$ if $Y$ is a $\G$-$C^*$-operator system and $\iota: X \to Y$ is a $\G$-$C^*$-equivariant uci map. The $\G$-$C^*$-extension $(Y, \iota)$ is called $\G$-$C^*$-rigid if it has the property that if $\phi: Y \to Y$ is a ucp map with $\phi \iota = \iota$, then $\phi = \id_Y$. It is called $\G$-$C^*$-injective extension of $X$ if $Y$ is $\G$-$C^*$-injective. A $\G$-$C^*$-injective extension of $X$ is called $\G$-$C^*$-injective envelope if the situation $\iota(X)\subseteq \widetilde{X}\subseteq Y$ with $\widetilde{X}$ a $\G$-$C^*$-injective operator subsystem  of $Y$ implies that $\widetilde{X}= Y$.

Similar definitions are made for $\wG$-operator systems.
\end{defn}

\begin{prop}\label{rigid crossed product}
    Let $(X, \alpha),(Y, \beta)$ be $\wG$-operator systems and $\iota: (X, \alpha) \to (Y, \beta)$ be an equivariant uci map. The following statements are equivalent:
    \begin{enumerate}
        \item $(Y, \iota: (X, \alpha) \to (Y, \beta))$ is a $\wG$-rigid extension of $X$.
        \item $(Y \rtimes_{r, \beta} \wG, \iota \rtimes_r \wG: (X \rtimes_{r, \alpha}\wG, \id \ovot \Delta) \to (Y \rtimes_{r, \beta} \wG, \id \ovot\Delta))$ is a $\G$-$C^*$-rigid extension of $X \rtimes_{r,\alpha} \wG$. 
    \end{enumerate}
\end{prop}
\begin{proof} $(1)\implies (2)$ Let $\Phi: (Y\rtimes_{r, \beta} \wG, \id\ovot \Delta) \to (Y\rtimes_{r, \beta}\wG, \id \ovot \Delta)$ be a $\G$-$C^*$-equivariant ucp map such that $\Phi\circ (\iota \rtimes_r \wG) = \iota \rtimes_r \wG$. We must show that $\Phi = \id$. By assumption, $\Phi(1\otimes \lambda(a)) = 1\otimes \lambda(a)$ for all $a\in \mathcal{O}(\G)$. Moreover, by equivariance of $\Phi$, we see that $\Phi(\beta(Y))\subseteq \beta(Y)$, so we can define the map $\phi: Y \to Y$ by the composition
$$
\begin{tikzcd}
Y \arrow[r, "\beta"] & \beta(Y) \arrow[rr, "\Phi"] &  & \beta(Y) \arrow[r, "\beta^{-1}"] & Y.
\end{tikzcd}$$
If $x\in X$, we have
$$\phi\iota(x) = \beta^{-1}\Phi\beta(\iota(x)) = \beta^{-1} \Phi (\iota \ovot \id)\alpha(x)= \beta^{-1} (\iota \ovot \id)\alpha(x)= \beta^{-1}\beta(\iota(x))= \iota(x) \ ,$$
so that $\phi \circ\iota =\iota$. Moreover, if $Y\subseteq B(\H)$, then $\Phi$ extends to a map $B(\H\otimes L^2(\G))\to B(\H\otimes L^2(\G))$ by Arveson's extension theorem, and this map is $\id \ovot \hat{\Delta}_r$-equivariant by Lemma \ref{multiplicative domain argument}. Therefore, also $\phi: (Y, \beta)\to (Y, \beta)$ is $\wG$-equivariant, and by $\wG$-ridigity of the extension, we conclude that $\phi = \id_Y$. In other words, $\Phi \beta = \beta$. If then $y\in Y$ and $a\in \mathcal{O}(\G)$, we see that
$$\Phi(\beta(y)(1\otimes \lambda(a)))= \Phi(\beta(y))(1\otimes \lambda(a)) = \beta(y)(1\otimes \lambda(a)) \ ,$$
so $\Phi = \id$ on $Y \rtimes_{\beta, r} \wG$. 

$(2)\implies (1)$ Assume that $\phi: (Y, \beta)\to (Y, \beta)$ is a $\wG$-equivariant ucp map such that $\phi \iota = \iota$.  Then clearly $(\phi \ovot \id)(\iota \rtimes_r \wG) = \iota \rtimes_r\wG$ so by $\G$-$C^*$-rigidity we see that $\phi \ovot \id = \id_{Y\rtimes_{\beta,r}\wG}$. We thus see that for $y\in Y$, $\beta(y) =(\phi \ovot \id)\beta(y)= \beta(\phi(y))$
so $\phi(y) = y$ by injectivity of $\beta$. Thus, $\phi =\id_Y$.
\end{proof}

Next, we establish existence and uniqueness of $\G$-$C^*$-injective envelopes.

\begin{prop}\label{injective envelope} Let $X$ be a $\G\text{-}C^*$-operator system. There exists a $\G\text{-}C^*$-injective envelope $(S, \iota)$ for $X$. If $(\widetilde{S}, \tilde{\iota})$ is another $\G\text{-}C^*$-injective envelope, there exists a unique $\G\text{-}C^*$-unital order isomorphism $\theta: S \to \widetilde{S}$ such that the diagram
$$
\begin{tikzcd}
S \arrow[rr, "\theta", dashed] &                                                    & \widetilde{S} \\
                               & X \arrow[lu, "\iota"] \arrow[ru, "\tilde{\iota}"'] &              
\end{tikzcd}$$
   commutes.
   Moreover, $(S, \iota: X \to S)$ is a $\G$-$C^*$-injective envelope if and only if $(S, \iota)$ is $\G$-$C^*$-injective and $\G$-$C^*$-rigid.
\end{prop}
\begin{proof} If $(S, \iota)$ is a $\G$-$C^*$-injective and $\G$-$C^*$-rigid extension of the $\G$-$C^*$-operator system $X$, then $(S, \iota)$ is a $\G$-$C^*$-injective envelope of $X$. Indeed, if $\iota(X)\subseteq T \subseteq S$ where $T$ is a $\G$-$C^*$-injective operator subsystem of $S$, there is a $\G$-$C^*$ ucp map $\theta: S \to T$ such that $\theta\vert_T = \id_T$. Viewing $\theta$ as a map $S \to S$ and noting that $\theta \iota = \iota$, the $\G$-$C^*$-rigidity shows that $\theta = \id_S$. Thus, $S = \theta(S)\subseteq T\subseteq S$ and $S=T$. 

We now use this fact to show that $\G$-$C^*$-injective envelopes exist.
We may assume that $X\subseteq B(\H)$. Let $Y:= \mathcal{R}(B(\H) \ovot B(L^2(\G)), \id \ovot \Delta)$ which is a $\G$-$C^*$-injective operator system and consider the set $\mathcal{G}$ of all $\G$-$C^*$-equivariant ucp maps $\varphi: Y \to Y$ such that $\varphi\circ \alpha = \alpha$. Viewing $\mathcal{G}$ as a subset of the ucp maps $Y \to B(\H \otimes L^2(\G))$, we observe that $\mathcal{G}$ is closed in the topology of pointwise $\sigma$-weak convergence. Indeed, assume that $\{\varphi_i\}_{i\in I}$ is a net in $\mathcal{G}$ that converges to the ucp map $\varphi: Y \to B(\H \otimes L^2(\G))$ in this topology. If $y$ is in the algebraic core of the $\G$-$C^*$-action $\id \ovot \Delta: Y \curvearrowleft \G$, then 
\begin{align*}
    (\id \ovot \Delta)\varphi(y)= \lim_{i\in I} (\id \ovot \Delta)\varphi_i(y) = \lim_{i\in I} (\varphi_i\odot \id)(\id \ovot \Delta)(y) = (\varphi \odot \id)(\id \ovot \Delta)(y) \ ,
\end{align*}
and we conclude that $\varphi(Y)\subseteq Y$ and $(\id \ovot \Delta)\varphi = (\varphi \otimes \id)(\id \ovot \Delta)$. Moreover, if $x\in X$, then
$$\varphi(\alpha(x)) = \lim_{i\in I} \varphi_i(\alpha(x))= \alpha(x) \ ,$$
so $\varphi\in \mathcal{G}$. By \cite[Proposition 2.1]{HHN}, it follows that there exists an idempotent $\varphi\in \mathcal{G}$ 
such that $\varphi = \varphi\phi \varphi$ for all $\phi\in \mathcal{G}$. Since $\varphi \alpha = \alpha$, we see that $\alpha(X)\subseteq \varphi(Y)$. We claim that $(\varphi(Y), \alpha: X \to \varphi(Y))$ is a $\G$-$C^*$-injective envelope of $X$. Consider the action $\beta:= \id \ovot \Delta: Y \to Y \otimes C_r(\G)$, which by $\G$-$C^*$-equivariance of $\varphi$ restricts to a map $\widetilde{\beta}:\varphi(Y)\to \varphi(Y)\otimes C_r(\G)$. In this way $(\varphi(Y), \widetilde{\beta})$ becomes a $\G$-$C^*$-operator system. Indeed, note that $$\widetilde{\beta}(\varphi(y))(1\otimes \lambda(g)) = (\varphi \otimes \id)(\beta(y))(1\otimes \lambda(g)) = (\varphi\otimes \id)(\beta(y)(1\otimes \lambda(g))$$
so the Podleś condition for $\varphi(Y)$ follows from the Podleś condition for $Y$.

Since $\varphi\circ \varphi = \varphi$, we see that $\varphi$ is the identity on $\varphi(Y)$. Hence, $\varphi(Y)$ is $\G$-$C^*$-injective.

Let $\psi: \varphi(Y)\to \varphi(Y)$ be a ucp $\G$-$C^*$-equivariant map with $\psi\circ \alpha = \alpha$. Then $\psi \varphi\in \mathcal{G}$ and $\varphi \psi = \psi$, so that
$$\varphi = \varphi(\psi\varphi)\varphi= \psi\varphi \ ,$$
which means that $\psi= \id.$ This establishes the $\G$-$C^*$-rigidity of the extension. The statement about uniqueness of the $\G$-$C^*$-injective envelope follows immediately from the definitions of $\G$-$C^*$-injectivity and $\G$-$C^*$-rigidity. 

Finally, if $(S, \iota)$ is a $\G$-$C^*$-injective envelope, then by the uniqueness it is canonically isomorphic to $(\varphi(Y), \alpha)$, which has the required equivariant rigidity and injectivity properties. \end{proof}

In view of the preceding proposition, we can speak about \emph{the} injective envelope of  a $\G$-$C^*$-operator system $X$, which we will denote by $I_{\G}^{C^*}(X)$. The $\G$-$C^*$ action on this space is implicitly understood.

\begin{rem}\normalfont{A similar result is of course true for $\wG$-operator systems. In this case, starting from a $\wG$-operator system $(X, \alpha)$ with $X\subseteq B(\H)$, one applies \cite[Proposition 2.1]{HHN} to the set $\mathcal{G}$ of all $\wG$-equivariant ucp maps $\varphi: B(\H) \ovot \mathscr{L}^\infty(\wG)\to B(\H) \ovot \mathscr{L}^\infty(\wG)$ satisfying $\varphi \circ \alpha = \alpha$ to find a minimal idempotent $\varphi\in \mathcal{G}$. One then checks that $$(\varphi(B(\H)\ovot \mathscr{L}^\infty(\wG)), \alpha: X \to \varphi(B(\H)\ovot \mathscr{L}^\infty(\wG)))$$ is a $\wG$-injective envelope for $X$. 

The only thing that may not be entirely obvious is that $\mathcal{G}$ is closed in the topology of pointwise $\sigma$-weak convergence. But suppose that $\varphi: B(\H)\ovot \mathscr{L}^\infty(\wG) \to B(\H)\ovot \mathscr{L}^\infty(\wG)$ is a ucp map and that $\{\varphi_i\}_{i\in I}$ is a net in $\mathcal{G}$ that converges to $\varphi$
in this topology. If $\omega \in \mathscr{L}^\infty(\wG)_*$ and $y\in B(\H)\ovot \mathscr{L}^\infty(\wG)$, we see that
\begin{align*}
    (\id \ovot \id \ovot \omega)((\id \ovot \hat{\Delta})\varphi(y))&= \lim_{i\in I} (\id \ovot \id \ovot \omega) ((\id \ovot \hat{\Delta})(\varphi_i(y)))\\
    &= \lim_{i\in I} (\id \ovot \id \ovot \omega)((\varphi_i \ovot \id)(\id \ovot \hat{\Delta})(y))\\
    &= \lim_{i\in I} \varphi_i((\id \ovot \id \ovot \omega)(\id \ovot \hat{\Delta})(y))\\
    &= \varphi((\id \ovot \id \ovot \omega)(\id \ovot \hat{\Delta})(y))\\
    &= (\id \ovot \id \ovot \omega)((\varphi \ovot \id)(\id \ovot \hat{\Delta})(y)) \ ,
\end{align*}
so we conclude that $(\id \ovot \hat{\Delta})\varphi = (\varphi \ovot \id)(\id \ovot \hat{\Delta})$, i.e. $\varphi$ is $\wG$-equivariant and thus $\varphi\in \mathcal{G}$.$\Di$}
\end{rem}

The $\wG$-injective envelope of a $\wG$-operator system $X$ will be denoted by $I_{\wG}(X)$ and the $\wG$-action on $I_{\wG}(X)$ is implicitly understood.

After our work, the following result comes for free:

\begin{thm}\label{description} Let $X$ be a $\wG$-operator system. Let $(Y, \iota)$ be a $\wG$-extension of $X$. Then $(Y, \iota)$ is the $\wG$-injective envelope of $X$ if and only if $(Y \rtimes_{r} \wG, \iota \rtimes_r \wG)$ is the $\G$-$C^*$-injective envelope of $X\rtimes_r \wG$. In particular,
$$I_{\wG}(X)\rtimes_r \wG = I_{\G}^{C^*}(X\rtimes_r \wG).$$
\end{thm}
\begin{proof}
    Immediate from Theorem \ref{main result}, Proposition \ref{rigid crossed product} and Proposition \ref{injective envelope}.
\end{proof}

Recall that the $C^*$-algebra $C(\partial_F \wG)$ of the Furstenberg boundary $\partial_F \wG$ can be realised as the operator system $I_{\wG}(\mathbb{C})$ \cite{KKSV}. In view of this, the following is obvious:
\begin{cor}  
    $I_\G^{C^*}(C_r(\G)) = C(\partial_F \wG)\rtimes_r \wG$ as $\G$-$C^*$-operator systems.
\end{cor}

\textbf{Acknowledgments:} Both authors would like to thank Kenny De Commer for valuable input throughout the entire project. The authors also want to thank Sergey Neshveyev for a useful discussion and Stefaan Vaes for sharing the proof of Proposition \ref{multiplicative} and allowing us to write it down in this paper.

The first named author would like to thank Adam Skalski for hospitality and a useful discussion
related to this project during a research visit to IMPAN in the framework of the FWO–PAS project VS02619N “von Neumann algebras arising from quantum symmetries”. This author is supported by Fonds voor Wetenschappelijk Onderzoek (Flanders), via an FWO
Aspirant-fellowship, grant 1162522N. The second author is supported by the NFR project 300837 “Quantum Symmetry”.
 \nocite{*}
\printbibliography[heading=bibintoc, title={References}] 
\end{document}